\documentclass{amsart}
\usepackage{amsfonts,amsmath,amssymb,amsthm,url}
\usepackage{xypic}
\usepackage{flafter}
\usepackage{hyperref}

\def\Tr{\text{\rm Tr}}
\def\GL{\text{\rm GL}}
\def\Sp{\text{\rm Sp}}
\def\Mat{\text{\rm Mat}}

\def\Hom{\text{\rm Hom}}
\def\End{\text{\rm End}}
\def\Aut{\text{\rm Aut}}
\def\Mat{\text{\rm Mat}}

\def\Fp{\mathbf{F}_p}

\def\id{\text{\rm id}}
\def\isar{\ \smash{\mathop{\longrightarrow}\limits^{\thicksim}}\ }

\newcommand{\M}{\mathcal{M}}  % submodule of K^2
\newcommand{\A}{\mathcal{A}}  % moduli space
\renewcommand{\O}{\mathcal{O}}  % order

\renewcommand{\Im}{\mathop{\mathrm{Im}}}
\renewcommand{\Re}{\mathop{\mathrm{Re}}}
\newcommand{\Z}{\mathbf{Z}}
\newcommand{\Q}{\mathbf{Q}}
\newcommand{\R}{\mathbf{R}}
\newcommand{\C}{\mathbf{C}}

\newcommand{\norm}{\mu}
\newcommand{\transpose}{^\mathrm{t}}
\newcommand{\graph}{\mathop{\mathrm{graph}}}
\newcommand{\modulinonred}{{\mathcal A}_2(l,l)} % moduli space of (l,l)-endo's
\newcommand{\moduli}{{\mathcal A}_2^{\mathrm{red}}(l,l)} % moduli space of (l,l)-endo's
\newcommand{\modulitwo}{{\mathcal A}_2^{\mathrm{red}}(2,2)} % moduli space of (2,2)-endo's
\newcommand{\abs}[1]{|#1|}
\numberwithin{equation}{section}

\newtheorem{theorem}{Theorem}[section]
\newtheorem{proposition}[theorem]{Proposition}
\newtheorem{example}[theorem]{Example}
\newtheorem{lemma}[theorem]{Lemma}

\newtheorem{remark}[theorem]{Remark}
\newtheorem{definition}[theorem]{Definition}

\begin{document}

\author{Reinier Br\"oker, Kristin Lauter, and 
        Marco Streng}
\thanks{
The authors thank David Gruenewald,
Damiano Testa and John Voight for helpful
discussions and the anonymous referee for suggestions to the improvement
of the exposition.
Part of the research for this paper was done during an internship
of the third-named author at Microsoft Research.
The third-named author is partially supported by
EPSRC grant number EP/G004870/1.
}
\address{Reinier~Br\"oker, Brown University, Department of Mathematics, Providence, RI, 02912, USA}
\email{reinier@math.brown.edu}
\address{Kristin~Lauter, Microsoft Research, One Microsoft Way, Redmond, WA, 98052, USA}
\email{klauter@microsoft.com}
\address{Marco~Streng, VU University Amsterdam, the Netherlands}
\email{marco.streng@gmail.com}
\urladdr{http://www.few.vu.nl/~streng/}


\title{Abelian surfaces admitting an $(l,l)$-endomorphism}

\begin{abstract}
We give a classification of all principally polarized abelian surfaces
that admit an $(l,l)$-isogeny to themselves, 
and show how to compute all the
abelian surfaces that occur.
We make the classification
explicit in the simplest case $l=2$.
As part of our classification, we also
show how to find all principally polarized abelian surfaces
with multiplication by a given imaginary quadratic order.
\end{abstract}

\maketitle

\section{Introduction}
\label{intro}

For a prime~$l$, the $l$-th \emph{modular polynomial}
$\Phi_l \in \Z[X,Y]$ is a singular
model for the modular curve $Y_0(l)$ parametrizing elliptic curves 
together with an isogeny of degree~$l$. This polynomial is used in various
algorithms, including the
`Schoof-Atkin-Elkies'
algorithm (see~\cite{hehcc17}) to count the
number of points on an elliptic curve~$E/\Fp$.
If we specialize $\Phi_l$
in $Y=X$ we get a \emph{univariate} polynomial $\Phi_l(X,X) \in \Z[X]$ whose
roots are the $j$-invariants of the elliptic curves $E/\C$ that admit an 
endomorphism of degree~$l$. There is a close link between $\Phi_l(X,X)$ and
the \emph{Hilbert class polynomials}~$H_\O$ of imaginary quadratic
orders~$\O$.  More precisely, we have
\begin{equation}\label{eq:product}
\Phi_l(X,X) = \prod_{\O} H_\O(X)^{e(\O)}
\end{equation}
where the product ranges over all imaginary quadratic orders~$\O$ that
contain an element of norm~$l$. The exponent
$$
e(\O) = \#\left( \{ \alpha \in \O \hbox{\ of norm\ } l\} / \O^*\right)  \in \Z_{>0}
$$
measures how many elements of norm~$l$ there are in~$\O$.
One of the applications of~\eqref{eq:product} is
a proof that $H_\O(X)$ has integral coefficients.

In this article we investigate an analogue of equation~\eqref{eq:product}
for abelian surfaces.
The appropriate analogue of an isogeny
of degree $l$ is an \emph{$(l,l)$-isogeny}.
If $A$ and $B$ are 
abelian surfaces with principal polarizations $\varphi_A$ and $\varphi_B$, then
an 
$(l,l)$-isogeny
$$
\lambda: (A,\varphi_A) \rightarrow (B,\varphi_B)
$$
is an isogeny $\lambda:A\rightarrow B$ such that
we have $\varphi_A\circ [ l ] = \widehat{\lambda}\circ \varphi_B
      \circ \lambda$, with $\widehat{\lambda}$ the dual of~$\lambda$.
{Let $\moduli \subset \A_2$ be the subset of principally polarized
abelian surfaces $(A,\varphi_A)$ admitting an $(l,l)$-endomorphism,
that is, an $(l,l)$-isogeny to itself.}
The set $\moduli$ is the natural analogue of the zero set of
$\Phi(X,X)$. The goal in this article is to identify the
irreducible subvarieties of~$\moduli$, which is the natural analogue of
equality~\eqref{eq:product}.

While the moduli space of elliptic curves is 1-dimensional, the moduli
space $\A_2$ of principally polarized abelian surfaces is $3$-dimensional.
In Section~\ref{sec:defofmoduli}, we will
define $\moduli$ as a reduced subscheme of $\A_2$ of dimension~$2$.
In particular, it
is a union of finitely many irreducible points, curves, and surfaces.
Our goal is to describe these subvarieties and their moduli interpretations,
and to list them explicitly.
There are various ways to describe varieties, and in this article we choose
\emph{two} descriptions. 

Firstly, we will describe most irreducible subvarieties of $\moduli$ by 
certain pairs
$$
(\M,T)
$$
consisting of a $\O$-submodule~$\M$ and an invertible $2\times 2$-matrix~$T$
with coefficients in the quotient field of~$\O$. Here, $\O$ ranges
over certain orders in degree two and four number fields. This description
dates back to Shimura, and we give a detailed exposition in Section~3. In
particular, Theorem~\ref{thm:modulispaces} gives the dimension of the modular variety
associated to a pair~$(\M,T)$, as well as references to explicit constructions
as complex analytic variety. In Section~4, we explain how to compute, for 
a given~$l$, all pairs $(\M,T)$. The only case where we cannot use
this description is where the ring $\Z[x]$ generated by 
an~$(l,l)$-endomorphism~$x$ is not a domain. We will show in Lemma~\ref{nonfieldlemma} that we can
describe all points in~$\moduli$ with this property by 
triples~$(E_1,E_2,\Gamma)$ consisting of elliptic curves~$E_1,E_2$ and
a finite subgroup of $E_1\times E_2$.
This completely identifies the
space $\moduli$ we are interested in.

Secondly, we can take a more explicit approach and try to describe
each irreducible subvariety by defining equations or a parametrization
of a subspace of an \emph{algebraic} model of~$\A_2$
(Section~5). Our methods depend on whether $(\M,T)$ corresponds
to a surface, curve or point inside the moduli space. As we will explain,
our approach need not work in general yet for pairs $(\M,T)$ corresponding
to \emph{points}. 

In Section~\ref{sec:22endomorphisms}, 
we will relate $\moduli$ to a 
computationally convenient genus-two
analogue of the modular polynomials. We will use this to make the
simplest case~$l=2$ completely explicit.
More precisely, we will prove the following theorem.

\begin{theorem}\label{leq2theorem}
The set $\modulitwo$ is the union of the Humbert surface of discriminant~$8$, 
which is irreducible, and a known finite set of CM-points.
\end{theorem}
For the list of CM-points, see Theorem~\ref{thm:summary},
of which Theorem~\ref{leq2theorem} is a summary.
We see that all the (Shimura) curves 
inside~$\modulitwo$ live on the Humbert surface. We expect that this `collapse'
only occurs for small~$l$ and there are Shimura curves not on a Humbert 
surface inside $\moduli$ for larger~$l$. 

\section{Background}
\label{background}

\subsection{Complex abelian surfaces and isogenies}

In this section we recall some background on complex 
abelian surfaces and $(l,l)$-isogenies.
We state the results
in terms of Riemann forms, but note that most can be generalized to fields
of arbitrary characteristic.

Let $V$ be a 2-dimensional complex vector space, and let
$\Lambda \subset V$ be a full lattice. Any form 
$E: \Lambda \times \Lambda \rightarrow \Z$ can be extended to a form
$E: V \times V \rightarrow \R$, and we call $E$
a \emph{Riemann form} if
the form $H(x,y) = E(ix,y) + iE(x,y)$ is Hermitian and positive
definite.
We say that the quotient $V/\Lambda$ is a (complex) \emph{abelian surface} if 
$\Lambda\subset V$ admits a Riemann form.

For an abelian surface $A = V/\Lambda$, we define its dual
as $A^\vee = 
V^\vee/\Lambda^\vee$ with
\begin{eqnarray*}V^\vee &=& \{ f : V \rightarrow \C \mid f(a v) = \overline{a} f(v),
             f(v+w) = f(v)+f(w) \},\\
  \Lambda^\vee &=& \{ f \in V^\vee \mid \Im(f(\Lambda)) \subset \Z \}.
\end{eqnarray*}
A Riemann form $E$ on $A$ determines a
homomorphism 
\begin{equation}\label{eq:polarization}
\varphi_E : A \rightarrow A^\vee
\end{equation}
by $\varphi_E(x) = H(x,\cdot)$ for $x \in V$.
The map $\varphi_E$ is called
a \emph{polarization} and the pair
$(E,\varphi_E)$ a \emph{polarized abelian variety}. 
We say that $(A,\varphi_E)$ is \emph{principally 
polarized} if $\varphi_E$ is an isomorphism,
or, equivalently, if $\det E = 1$ on $\Lambda\times\Lambda$.
In this case, the form~$E$ defines a unique ring
involution $'$
on $\End(A)$ 
via $E(yu,v) = E(u,y'v)$ for $y\in\End(A)$. 

\begin{definition} The map $'$ above is called the
\emph{Rosati-involution} associated to the
principal polarization~$\varphi_E$.
\end{definition}

A surjective morphism $f: A \rightarrow B$ between abelian surfaces is called 
an \emph{isogeny} if it has finite kernel.
The polarization $\varphi_E$ occuring
in~\eqref{eq:polarization} is an example of an isogeny.
If $E'$ is a Riemann form
on $B$, then an isogeny $f: A \rightarrow B$ induces a
Riemann form 
$$
f^*E': (u,v) \mapsto E'(f(u),f(v))
$$
on $A$.

\begin{definition} For a positive integer~$n$, an $(n,n)$-isogeny $f: (A,E) \rightarrow
(B,E')$ is an isogeny $A \rightarrow B$ such
that
\begin{enumerate}
\item $f^* E' = n E$ holds and
\item the kernel of $f$ is contained in~$A[n]$.
\end{enumerate}
\end{definition}
For prime~$n$, the second item is automatic from
the first, and the definition is easily seen to
be equivalent to the definition given in
the introduction. 

\begin{lemma}\label{lem:nntoself} Let $(A,\varphi_E)$ be a principally polarized abelian surface that
admits an $(n,n)$-isogeny $x$ to itself. Then we have  
$x' x = [n]$.
\end{lemma}
\begin{proof} This follows from the fact that the
equality $E(u,x'xv) = E(xu,xv) = n E(u,v) = E(u, n v)$
holds for all $u,v$. 
\end{proof}

The kernel of an $(n,n)$-isogeny $x: (A,E) \rightarrow (B,E')$ is a subgroup
of the $n$-torsion~$A[n]$ with additional structure. To analyze it,  
we define 
the \emph{Weil pairing}
$e_n:A[n] \times A[n] \rightarrow \C^*[n]$ by
$$
e_n(u, v) = \exp(2\pi i n E(u, v))
$$
for $u$, $v \in \frac{1}{n}\Lambda \subset V$.
We say that a subgroup $G \subset A[n]$ is 
\emph{isotropic} (with respect to the Weil pairing)
if $e_n$ restricts to
the trivial form on $G \times G$. We say that $G$ is
\emph{maximally isotropic}
if it is not strictly contained in another isotropic subgroup.

\begin{lemma}\label{isotropicWeil}
Let $x: (A,E) \rightarrow (B,E')$ be an 
$(n,n)$-isogeny. Then the kernel of $x$ is maximally isotropic with respect
to the Weil pairing. Furthermore, every maximally isotropic subgroup 
$G \subset A[n]$ arises as the kernel of an $(n,n)$-isogeny,
which is unique up to composition with an isomorphism of principally
polarized abelian surfaces.

If $n$ is prime, then there are
exactly $(n^4-1)/(n-1)$ such subgroups.
\end{lemma}
\begin{proof} The first two statements follow from~\cite[Prop.\ 16.8]{milne}.
The last statement follows from the proof of~\cite[Lemma 6.1]{brokerlauter}.
\end{proof}

We will need a few properties of the analytic and rational representation
of an $(n,n)$-isogeny. We recall the basic concepts in the remainder of
this section.
Let $A_1 = \C^2/\Lambda_1$ and $A_2 = \C^2/\Lambda_2$ be two principally 
polarized abelian surfaces,
and let $f: A_1 \rightarrow A_2$ be a homomorphism.
We can uniquely lift $f$ to a $\C$-linear map
$F: \C^2 \rightarrow 
\C^2$ with $F(\Lambda_1) \subseteq \Lambda_2$. The natural map
$$
\rho_{\mathrm{a}}: \Hom(A_1, A_2) \rightarrow \Hom_\C(\C^2,\C^2)
$$
sending $f$ to $F$ is called the
\emph{analytic representation}
of 
$\Hom(A_1,A_2)$. By restricting $F$ to the lattice $\Lambda_1$, we get a 
$\Z$-linear map $F_{\Lambda_1}$ that induces~$F$. The natural map
$$
\rho_{\mathrm{r}}: \Hom(A_1, A_2) \rightarrow \Hom_\Z(\Lambda_1,\Lambda_2)
$$
sending $f$ to $F_{\Lambda_1}$ is called the
\emph{rational representation}
of $\Hom(A_1,A_2)$. We use the same notation and terminology for
the $\Q$-linear extension of $\rho_{\mathrm{a}}$ and~$\rho_{\mathrm{r}}$.

For $f \in \End_\Q(A) := \End(A) \otimes \Q$,
let $P_f^{\mathrm{a}}$
be the characteristic polynomial
of the analytic 
representation $\rho_{\mathrm{a}}(f)$ of~$f$.
Likewise, we get $P_f^{\mathrm{r}}$ for the rational 
representation of~$f$.
Explicitly, we have
$$
P_f^{\mathrm{a}} = \det(X\thinspace \id_{\C^2} - \rho_{\mathrm{a}}(f))
   \in \C[X] \quad \hbox{and} \quad
P_f^{\mathrm{r}} = \det(X\thinspace \id_\Lambda - \rho_{\mathrm{r}}(f))
  \in \Q[X].
$$
The polynomial $P_f^{\mathrm{a}}$ is quadratic, and $P_f^{\mathrm{r}}$ has 
degree~$4$. The following lemma 
gives the relation between the analytic and the rational representation.

\begin{lemma}\label{theResultsFromSection2onPaAndPr}
Let $A = \C^2/\Lambda$ be a principally polarized abelian 
surface, and let $f \in \End_\Q(A)$ be given. Then we have\par
\smallskip
\item{(a)} $P_f^{\mathrm{r}} =
P_f^{\mathrm{a}} \cdot \overline{P_f^{\mathrm{a}}}$ and \label{rationalrepisproduct}
\item{(b)} $P_{f'}^{\mathrm{a}} = \overline{P_f^{\mathrm{a}}}$ \label{complexreprosati}\par\smallskip\noindent
where $\overline\cdot $ denotes complex conjugation and $'$ denotes the Rosati-involution.
The polynomial $P^{\mathrm{r}}_f$
has integer coefficients if $f$ is an endomorphism of~$A$.
\end{lemma}
\begin{proof} This follows immediately from~\cite[Prop.\ 5.1.2]{birkenhakelange} and~\cite[Lemma~5.1.4]{birkenhakelange}.
\end{proof}

\subsection{The definition of {$\moduli$}}\label{sec:defofmoduli}

By Lemma~\ref{isotropicWeil}, the $(l,l)$-isogenies $\lambda:(A,E)\rightarrow (B,E')$
correspond to pairs $((A,E), G)$, where $G\subset A[l]$
is a maximally isotropic subgroup.

Let $\Sp_{4}(\Z)\subset \GL_{4}(\Z)$ be the symplectic group,
consisting of the matrices $M=(a,b;c,d)$ with $a,b,c,d\in\Mat_2(\Z)$
such that
$ad\transpose - bc\transpose = 1$ holds and both $ab\transpose$ and $cd\transpose$ are symmetric.
Let $\mathcal{H}_2$ consist of all symmetric $2\times 2$ complex matrices
with positive definite imaginary part.
Then $\Sp_4(\Z)$ acts on $\mathcal{H}_2$ by
$M\tau = (a\tau+b)(c\tau+d)^{-1}$,
and $\A_2$ is the quotient (\cite[Chapter 8]{birkenhakelange}).

Let $\Gamma_0(l)\subset\Sp_4(\Z)$ be the set of matrices with
$C \equiv 0\bmod N$. Then $\A_2(\Gamma_0(l))=\Gamma_0(l)\backslash \mathcal{H}_2$
parametrizes the pairs $((A,E), G)$, i.e., the $(l,l)$-isogenies.
It has two natural maps to $\A_2$, given by
$\tau \mapsto \tau$ (corresponding to $\lambda\mapsto (A,E)$)
and $\tau\mapsto l\tau$ (corresponding to $\lambda\mapsto (B,E')$).

In particular, we get a natural map $\A_2(\Gamma_0(l))\rightarrow \A_2\times \A_2$.
Both $\A_2$ and $\A_2(\Gamma_0(l))$ are quasi-projective
algebraic varieties (as in \cite[Remark 8.10.4]{birkenhakelange}),
so the image of $\A_2(\Gamma_0(l))$ in $\A_2\times \A_2$ is a
subvariety.
The intersection with the diagonal $\Delta:\A_2\rightarrow \A_2\times \A_2$ gives us a subscheme $\modulinonred = \Delta^{-1}(\A_2(\Gamma_0(l)))$
of $\A_2$ of which we take the reduced subscheme~$\moduli$.
As $\moduli\subsetneq\A_2$ is a reduced subscheme,
it is a union of finitely many irreducible points,
curves and surfaces.

\section{Moduli spaces of abelian surfaces with endomorphism structure}
\label{sec:endomrings}
For a principally polarized abelian surface $A = (A,\varphi_E)$ with 
$(n,n)$-isogeny~$x$ to itself, the algebra $K = \Q[x]$ is a subalgebra of the
endomorphism algebra~$\End_\Q(A)=\End(A)\otimes\Q$.
We distinguish between
the case that $K$ is / is not a field.

\subsection{$K$ is not a field}
We start with the case where $K$ is \emph{not} a field,
as we will see that this is the only case where Shimura's
classification of the moduli spaces does not apply.
The 
result is the following lemma.

\begin{lemma}
\label{nonfieldlemma}
  Let $x$ be an endomorphism of a principally polarized abelian surface
  $A/\C$ with $xx' = l$ for some prime~$l$. 
  Let $K=\Q[x]\subseteq \End_\Q(A)$ and assume that $K$ is not
a field. Then we have
  $P_x^{\mathrm{a}}=(X-\beta_1)(X-\beta_2)$
     for some pair of imaginary quadratic integers
     $\beta_1, \beta_2\in\C$ of absolute value $\sqrt{l}$ and non-equal trace.
The map
$$
x \mapsto (\beta_1,\beta_2)
$$
gives an isomorphism $K \isar \Q(\beta_1) \times \Q(\beta_2)$.
Furthermore, there exist elliptic curves $E_1,E_2$, a
positive integer $n \mid (\Tr \beta_1 - \Tr \beta_2) \not = 0$,
an $(n,n)$-isogeny $$ \lambda:E_1\times E_2\rightarrow A,
$$ and for $i=1,2$ an identification of $\Z[\beta_i]$ with
a subring of $\End(E_i)$ such that
$$x\circ \lambda=\lambda\circ (\beta_1, \beta_2).$$
\end{lemma}
If the ring $\Z[\beta]$ is contained in the endomorphism ring
of~$E$, then we automatically have $E \cong \C/\mathfrak{a}$ for
a $\Z[\beta]$-ideal~$\mathfrak{a}$.
It is well understood how to compute an algebraic model of 
$\C/\mathfrak{a}$,
we refer to~\cite{hehcc5} for details.
Once the finitely many candidates for $E_1$ and~$E_2$ are
computed, there are only finitely many possibilities for $n$
and $\Gamma=\ker(\lambda)$, hence for~$A$.
We represent~$A$ by the triple $(E_1,E_2,\Gamma)$.

\begin{proof}[Proof of Lemma~\ref{nonfieldlemma}]
Choose a basis of $\C^2$ such that the analytic representation of
$x$ is upper-triangular, say
$$ 
\rho_{\mathrm{a}}(x) = \left(\begin{array}{cc}\beta_1 & * \\ 0 & \beta_2\end{array}\right)
$$
with $\beta_1,\beta_2$ in $\C$.
Note that $\rho_{\mathrm{a}}(x')=\rho_{\mathrm{a}}(lx^{-1})$
is also upper-triangular, say
$$ 
\rho_{\mathrm{a}}(x') = \left(\begin{array}{cc}\gamma_1 & * \\ 0 & \gamma_2\end{array}\right).
$$

As $\beta_1$ and $\beta_2$ are roots of $P_{x}^{\mathrm{r}}$,
they are algebraic integers. We let
$g$ be the minimal polynomial of~$\beta_1$.
We claim that $\gamma_2$ is not an algebraic conjugate of $\beta_1$, i.e., that $g(\gamma_2)\not=0$ holds.
Indeed: otherwise we have $\Tr(g(x)g(x)')=\Tr(g(x)g(x'))=0$ and
because $'$ is positive definite by~\cite[Thm.~5.1.8]{birkenhakelange},
we find $g(x)=0$, contradicting the assumption that $K$ is not a field.

By Lemma~\ref{complexreprosati}(b), we have $\{\gamma_1,\gamma_2\}=\{\overline{\beta_1},\overline{\beta_2}\}$,
so the claim above yields
$\gamma_1=\overline{\beta_1},\gamma_2=\overline{\beta_2}$,
and that $\beta_1$ and $\beta_2$ are not algebraic conjugates.
It follows from $\beta_i\overline{\beta_i}=\beta_i\gamma_i=l$
that $\beta_i$ has absolute value $\sqrt{l}$ for $i=1,2$.
Now $\beta_1$ and
$\beta_2$ are non-conjugate imaginary quadratic integers
with the same norm, so they have distinct trace.
It follows that $x \mapsto (\beta_1,\beta_2)$ gives an isomorphism 
$$
K=\Q[X]/(P_x^{\mathrm{r}}) \isar \Q(\beta_1)\times \Q(\beta_2).
$$

To prove the remainder of the lemma, put 
$$
e_1 = {x+x'-\Tr\beta_2} \in \End(A),
$$
let $k = \Tr(\beta_1)-\Tr(\beta_2) \not = 0$, and $e_2=k-e_1$.
The relation $e_i^2 = ke_i$
holds on both $\beta_j$-eigenspaces $V_j\subset \C^2$,
and hence on~$A=\C^2/\Lambda$. We see 
that $f_i \in \End_\Q(A)$ satisfies
$$
f_i^{\prime} = f_i \qquad \hbox{and} \qquad f_i^2 = f_i,
$$
that is, $f_i$ is a symmetric idempotent in the endomorphism algebra. Let 
$n \mid k$ be the smallest positive integer such that $nf_1$ (equivalently $nf_2$)
is an 
endomorphism
and let $E_i$ be the image of~$nf_i$,
which is an elliptic curve.
We consider the natural summation map $$\lambda: E_1\times E_2\rightarrow A,$$
which is surjective, hence an isogeny.
Moreover, as $E_i=V_i/(V_i\cap\Lambda)$,
we can identify $x_{|E_i}\in\End(E_i)$
with $\beta_i\in\Z[\beta_i]$,
which gives us $x\circ \lambda = \lambda\circ (\beta_1,\beta_2)$.
It now remains only to prove that $\lambda$ is an $(n,n)$-isogeny.

Let $\nu = (nf_1,nf_2):A\rightarrow E_1\times E_2$ and note $\lambda\circ\nu=nf_1+nf_2=n$,
so the kernel of the isogeny $\lambda$ is contained in $(E_1\times E_2)[n]$.
It now suffices to show that the induced polarization $\lambda^* E_A$
is $n$ times the standard polarization
on $E_1\times E_2$.
This is a direct application
of \cite[Corollary~5.3.6 and Criterion~5.3.4]{birkenhakelange}.
% In spite of 5.3.6 being a corollary, it is not appropriate
% to use the name ``Poincare's reducibility theorem''
% here, because showing that it is an isogeny was trivial,
% and that is what that theorem does. What we use is something
% from the proof of Corollary 5.3.6, together with some ideas
% from the proofs of Criterion 5.3.4 and Theorem 5.3.2.
For completeness, we give a few more details.

It follows from the easy fact that $V_1$ and $V_2$
are orthogonal for the Riemann form~$E_A$
(as $n^2E(v_1, v_2)= E(f_1v_1,f_2v_2)=
E(f_1f_2'v_1,v_2)=E(0,v_2)=0$ for $v_i\in V_i$)
that the induced polarization on $E_1\times E_2$
is a product of polarizations on $E_1$ and~$E_2$.
As every alternating form on $\Z^2$ is a multiple
of a form of determinant one, we find that
the induced polarizations on $E_1$ and $E_2$
are respectively
$d_1$ and $d_2$ times the unique principal polarizations
for positive integers~$d_i$.
We compute $d_i$ with \cite[Criterion~5.3.4]{birkenhakelange},
% That criterion uses the existence and basic properties of
% the norm element. I don't see (yet?) how to do this without
% the norm element, and I don't think we should explain
% what it is. So maybe this last sentence
% cannot be made any more self-contained.
which states that as $E_A$ is principal,
the number $d_i$ is the minimal integer
such that $d_i$ times the idempotent $e_i/k$ is an endomorphism,
that is, $d_i=n$.
% Proving $d_1*d_2 = degree(lambda) | n^2$ is relatively
% easy. Then proving n | d_i obviously requires minimality of
% n. How can that be used (other than using the norm element?)?
% The existence of the norm element shows d_i e_i/k is
% an endomorphism, and that is sufficient.
% So what we would need for a completely self-contained
% proof is a direct and elementary proof that
% d_i e_i / k maps Lambda to itself.
\end{proof}

\subsection{$K$ is a field}\label{sec:fieldcase}
\label{maintheoremsec}
We continue with the case that $K = \Q[x]$ is a field. The following
lemma restricts the structure of~$K/\Q$.
\begin{lemma}
\label{fieldcase}
Let $A$ and $K=\Q[x]$ be as above. Suppose that~$K$ is a field.
Then~$K$ is either a quadratic field or a 
degree-$4$ CM-field, that is, a degree-4 number field~$K$
such that complex conjugation induces the same
non-trivial automorphism
of~$K$ for every embedding $K\rightarrow\C$.
In both cases, the Rosati-involution~$'$ is complex conjugation on~$K$,
and~$x$ is a \emph{Weil $l$-number}, that is,
$x$ satisfies $x\overline{x}=l$ for every embedding into~$\C$.
\end{lemma}
\begin{proof} See~\cite[Sec.\ 5.5]{birkenhakelange},
except for the fact that~$x$ is a Weil number, which is
Lemma~\ref{lem:nntoself}.
\end{proof}

\begin{example}
\label{quadraticfields}
Let $\Q[x]$ be a quadratic field over $\Q$ and suppose that
we have $x\overline{x}=l$.
Then the trace of $x$ is at most $2\sqrt{l}$.
In particular, for $l=2$, we have $x\in\{\sqrt{-2},\frac{1}{2}(1+\sqrt{-7}),1+\sqrt{-1},\sqrt{2}\}$
up to sign and complex conjugation.
\end{example}

By Lemma~\ref{fieldcase}, the field $K$ is
quadratic or a degree-$4$ CM-field.
In the former case, let $m(K)=2$ and in the latter case let $m(K)=1$.
Writing $A = V/\Lambda$, there exists an isomorphism
$$
\Lambda \otimes \Q \isar K^{m(K)}.
$$
The lattice $\Lambda$ has a natural $\Z[x]$-module structure,
and, under an isomorphism as above,
corresponds to a $\Z[x]$-submodule~$\M$ of 
$K^{m(K)}$ of rank $4$ over~$\Z$.

\begin{lemma}
\label{easypol}
Let $A,K,\,',\M$ be as above. Then there exists a matrix $T 
\in \GL_{m(K)}(K)$ satisfying $(T')\transpose = -T$ such that 
$E(u,v) = \Tr_{K/\Q}(u\transpose T v')$ is the Riemann form inducing the 
principal polarization on~$A$. 
\end{lemma}
\begin{proof} See~\cite[Prop.\ 9.2.3.]{birkenhakelange} 
and~\cite[Prop.\ 9.6.5]{birkenhakelange}.
\end{proof}
\noindent
\begin{example}\label{ex:realquad}
If $K$ is a real quadratic field, then $'$ is the 
identity on~$K$. The matrix $T$ therefore equals
$$\Bigl(\begin{array}{cc} 0 & r \\ -r & 0\end{array} \Bigr)$$ for some $r \in K^*$.
\end{example}
\noindent
Lemma~\ref{easypol} tells us that the
abelian surface $A$ yields a pair $(\M,T)$ 
consisting of a $\Z[x]$-submodule $\M$ of $K^{m(K)}$ and a 
matrix $T \in \GL_{m(K)}(K)$ such that the $\Q$-bilinear form
$K^{m(K)}\times K^{m(K)}\rightarrow\Q$ given by
$(u,v)\mapsto \Tr_{K/\Q}(u\transpose T v')$
restricts to a $\Z$-bilinear form $\M\times\M\rightarrow\Z$
of determinant~$1$. 
Conversely, for a fixed algebraic integer $x$ and for a fixed pair
$(\M,T)$, Shimura gives
an explicit complex analytic description~\cite[Sec.\ 9.8]{birkenhakelange}
of the subspace $S(\M,T)\subset \A_2$
of principally polarized abelian
surfaces belonging to this pair $(\M,T)$.
The following theorem states that $S(\M, T)$ is an irreducible
variety and gives its dimension.

\begin{theorem}\label{thm:modulispaces}
Let $K$ be either a quadratic field or a degree-$4$ CM-field and
let $\Z[x]$ be an order in~$K$. Let $\M$ be a $\Z[x]$-submodule of $K^{m(K)}$ 
of rank~$m(K)$ and $T\in \GL_{m(K)}(K)$ a matrix satisfying $(T')\transpose = -T$.
Suppose that the $\Q$-bilinear form
$K^{m(K)}\times K^{m(K)}\rightarrow\Q$ given by
$(u,v)\mapsto \Tr_{K/\Q}(u\transpose T v')$
restricts to a $\Z$-bilinear form $\M\times\M\rightarrow\Z$
of determinant~$1$. 
Let $S(\M,T)\subset \A_2$ be the subspace of all principally polarized abelian surfaces
that correspond to $(\M,T)$.
Then the following holds:
\begin{itemize}
\item if $K$ is real quadratic, then $S(\M,T)$ is an
irreducible (Humbert) surface,
\item if $K$ is imaginary quadratic and $\det(T) > 0$, then $S(\M,T)$ is
   an irreducible (Shimura) curve.
   For any abelian surface $A$ in $S(\M,T)$,
   the polynomial $P_x^{\mathrm{a}}$ is the minimal polynomial of~$x$;
\item if $K$ is imaginary quadratic and $\det(T) <0$,
      then $S(\M,T)$ consists of one abelian surface $A$, which is
   isogenous as an unpolarized 
   abelian variety to a product $E\times E$ of an elliptic curve with itself.
   In this case we have $\End_\Q(E) = K$ and $P_x^{\mathrm{a}}=(X-\beta)^2$
   for some $\beta\in \C$;
\item if $K$ is a degree-$4$ CM-field, then $S(\M,T)$ 
   consists of one abelian surface.
\end{itemize}
\end{theorem}
\begin{proof} 
All cases are special cases of the theory in~\cite[\S 9.6]{birkenhakelange},
except the real quadratic case,
which is a special case of~\cite[\S 9.2]{birkenhakelange}.
\end{proof}

\section{Finding all irreducible subvarieties}\label{sec:finding}
Two pairs $(\M_1,T_1)$ and $(\M_2,T_2)$ as in Theorem~\ref{thm:modulispaces}
are called \emph{isomorphic} if there exists a matrix $S \in \GL_{m(K)}(K)$ with
$$
S^{-1} \M_1 = \M_2 \qquad \hbox{and} \qquad S\transpose T_1 S' = T_2.
$$
Isomorphic pairs yield the same moduli subspace inside~$\A_2$, and in this
section we explain how to find all isomorphism classes of pairs $(\M_i,T_i)$.

Moreover, if $\sigma$ is an automorphism of~$K$, then $(\M,T)$ 
and $(\sigma M, \sigma T)$ yield the same moduli space.
Hence, given $K$ and $x \in K$ we can describe all principally polarized 
abelian surfaces $A/\C$ admitting an endomorphism $x$ with $\Q[x]=K$ by
listing the isomorphism classes of pairs~$(\M_i,T_i)$
up to this action of $\mathrm{Aut}(K)$.

First assume that $K = \Q[x]$ is a quadratic field. Depending on~$K$ 
and~$(\M,T)$ we either get a surface, a curve or a point in the moduli space.

\subsection{Real quadratic fields}

\begin{proposition}\label{prop:humirr}
{Let $l$ be prime. Then the variety $\moduli$ contains 
exactly two irreducible Humbert surfaces for $l \equiv 1 \bmod 4$ and 
exactly one irreducible Humbert surface otherwise.}
\end{proposition}
\begin{proof}
This result is the special case $t=1$ of~\cite[Thm.~IX.2.4]{vdgeer},
which states that
the Humbert surface $H_D$ of discriminant~$D$ is irreducible for
every quadratic discriminant~$D$.
For $D=4l$ prime with $l\equiv 1\pmod{4}$, this gives one component $H_{D/4}$
and one component $H_D$. For $D\equiv 4l$ with
$l \not\equiv 1 \pmod{4}$, we have
only the irreducible surface~$H_D$.
\end{proof}

Alternatively, and to illustrate how one can use the theory
of Section~\ref{sec:fieldcase}, we give a self-contained proof
of Proposition~\ref{prop:humirr}.
Note that our claimed number of irreducible Humbert surfaces
equals the number of orders containing $\Z[\sqrt{l}]$.
Therefore, the proposition follows from Theorem~\ref{thm:modulispaces}
and the following lemma.

\begin{lemma}\label{bassring}\label{lem:humbertsurface}
{Let $K=\Q[x]$ be real quadratic and let $\O = \Z[x]$ be an order in~$K$. 
Let $(\M,T)$ be a pair consisting of an $\O$-submodule $\M$ of~$K^2$ 
that is free of rank 4 as a $\Z$-module and a matrix~$T \in \GL_2(K)$ that satisfies 
the conditions of Lemma~\ref{easypol}. Then, up to isomorphism, 
we have
$$
\M=\O_1\times \O_1 \quad \hbox{and} \quad
T = \delta(\O_1)^{-1}\left( \begin{array}{cc} 0 & 1 \\ -1 & 0\end{array}\right),
$$
where $\O_1\supset \O$ is an order of $K$ and $\delta=\delta(\O_1) = 2\alpha-
\Tr\alpha$ generates the different of~$\O_1~=~\Z[\alpha]$.}
\end{lemma}
\begin{proof} 
Let $(\M,T)$ be a pair consisting of an $\O$-module of $\Z$-rank 4 and a 
matrix $T$ satisfying the conditions of Theorem~\ref{thm:modulispaces}.
As $\O$ is quadratic, every
ideal is generated by 2 elements as a $\Z$-module, hence in particular as an $\O$-module.
By Bass~\cite[Prop.~1.5]{bass4}, this implies that every finitely generated projective $\O$-module
is the direct sum of ideals of $\O$.
For $\M$ of rank 4 over $\Z$, this implies (after a suitable choice of basis)
$\M = \mathfrak{a} \times 
\mathfrak{b}$ for fractional
(not neccessarily invertible)
$\O$-ideals $\mathfrak{a},\mathfrak{b}$. 

By Example~\ref{ex:realquad}, we know that $T$ is of the form 
$$\left(\begin{array}{cc} 0 & c \\ -c & 0\end{array}\right)$$ for some~$c\in K^*$.
We derive that $\mathfrak{a}$ has trace dual $\mathfrak{a}^\dagger = 
\mathfrak{b} c$, and as a 
consequence we see that $\mathfrak{a}$ and $\mathfrak{b}$ have the same
multiplier ring~$\O_1$. An argument similar to~\cite[Prop.~3.4]{LaANT} shows
that we have $\mathfrak{a}\mathfrak{a}^\dagger=\delta(\O_1)^{-1}$,
and $\mathfrak{b}$ therefore equals $(\delta(\O_1)c\mathfrak{a})^{-1}$ 
as $\O_1$-ideal.
Without loss of generality, we may therefore assume that
we have $\mathfrak{a} = \O_1$ and hence $\mathfrak{b}c = \O_1^\dagger = 
(1/\delta(\O_1)) \O_1$.
Then after scaling we get
$\M = \O_1 \times \O_1$ and $c = 1/\delta(\O_1)$,
which proves the lemma.
\end{proof}

\subsection{Imaginary quadratic fields}\label{ssec:imagquad}
The situation is similar for \emph{imaginary} quadratic fields~$K$: we again
have to list all pairs~$(\M,T)$ for the order~$\O = \Z[x]$. We may proceed
analogously to the proof of Lemma~\ref{bassring} and write $\M = \mathfrak{a}
\times \mathfrak{b}$ for fractional~$\Z[x]$-ideals~$\mathfrak{a},
\mathfrak{b}$. However, since we have less control over the matrix~$T$, we 
cannot derive that $\M = \O_1\times \O_1$ for an order~$\O_1\supset \O$ in this case.

For $\O = \Z[x]$, let $\delta(\O) = 2x -\Tr(x)$
and let $d(\O)=N_{K/\Q}(\delta(\O))$ be the discriminant.
Then, for the matrix $S=\delta T$, the 
condition $\overline{T}\transpose=-T$ is equivalent to
$$
\overline S\transpose = S,
$$
in other words, to~$S$ being \emph{Hermitian}.
There seems to be a lot of theory related to the classification of
pairs $(\M,S)$ that one could try to apply to the problem of enumerating
these pairs. For example, Hayashida-Nishi~\cite{hayashidanishi}
computes the class number of pairs $(\M, S)$ with $S$ positive
definite, Shimura~\cite{shimura-unitary} computes
the class number of pairs where $\M$ is a so-called
``maximal lattice'', and we will see below how to associate to $S$
the structure of a quaternion algebra on $\M\otimes \Q$, making
$\M$ into an ideal for an order $B$ in this quaternion algebra,
and one might hope to use
results on class numbers of~$B$. However,
such results all require $\O$ (or possibly even $B$) to be a maximal order.
General results that do not require maximal orders seem to be lacking,
hence we will give a more direct approach below. 
\begin{lemma}\label{lem:detisone}
 If $\M\supset \O\times \O$, then we have $[\M:\O\times \O]=\pm \det S$.
\end{lemma}
\begin{proof}
The 
determinant of $E:(u,v)\mapsto\Tr_{K/\Q}(u\transpose Tv')$
with respect to a $\Z$-basis of $\mathcal{O}\times\O$
can be computed to be
$(\det S)^2$.
With respect to a $\Z$-basis of $\M$, it therefore 
equals $[\M : \O\times\O]^{-2}(\det S)^2$.
On the other hand, as the polarization is principal,
this determinant is known to be~$1$, which
proves Lemma~\ref{lem:detisone}.
\end{proof}

Applying complex conjugation to $(\M,T)$ if needed, we may assume that $S$
is not negative definite.
Write $\Phi(u,v)=u\transpose S v'$ and $\norm(u) = \Phi(u,u)$.
We distinguish two cases: the case where $\norm$ takes the value
$0$ on $\M$ and the case where it does not.

\subsubsection*{If $\mu$ does not take the value~$0$.}\label{sssec:notvalue0}
Take a $K$-basis $b_1$, $b_2$ of $\M\otimes \Q$ with $b_1$, $b_2\in \M$,
and choose this basis such that $\abs{\mu(b_1)}$ is minimal
over all $b_1\in \M$,
and $\abs{\mu(b_2)}$ is minimal among all $b_2\in \M$ that
satisfy $\Phi(b_1,b_2)=0$.

Let $a=\norm(b_1)$ and $b=\norm(b_2)$, and identify
$\M\otimes \Q$ with $K^2$
via the basis $b_1$,~$b_2$.
\begin{lemma}\label{lem:boundsforthatbasis}
We have
\begin{enumerate}
\item\label{itm:boundsforthatbasis2} $$\mathcal{O}\times\mathcal{O}\subset \M \subset (\frac{1}{a}\mathcal{O})\times(\frac{1}{b}\mathcal{O}),$$
where for both inclusions the
index is $\abs{ab}$.
\item\label{itm:boundsforthatbasis1} $$1\leq |a|\leq
\frac{2\sqrt{t}}{\pi} {|d(\O)|^{1/2}}
\quad\mbox{and}\quad |a|\leq |b|\leq \frac{2}{\pi}|d(\O)|^{1/2}|a|,$$
with $t=1$ if $S$ is indefinite, and $t=2$ if $S$ is definite.
\end{enumerate}
\end{lemma}
\begin{proof}
Note that $b_1=(1,0)\transpose$ and $b_2=(0,1)\transpose$
are in the $\mathcal{O}$-module $\M$, hence $\M\supset \mathcal{O} \times\mathcal{O}$.
Now take any element $(u, v)\transpose$ in~$\M$. Then we have 
$au = \Phi(u,b_1)\in\mathcal{O}$ and $bv = \Phi(u,b_2)\in\mathcal{O}$,
which proves the inclusions in~(\ref{itm:boundsforthatbasis2}).
The index of both inclusions is $\abs{ab}$ by Lemma~\ref{lem:detisone}.

Let $V\subset \C^2$ be the set of vectors $(u,v)\transpose$ satisfying
\begin{align*}
\max\{|a| u\overline{u}, |b|v\overline{v}\} \leq B &\quad
\text{if $S$ is indefinite, and}\\
a u\overline{u} + b v\overline{v} \leq B &\quad
\text{if $S$ is positive definite}.
\end{align*}
It is closed, convex, and symmetric, and has volume $s \pi^2 B^2|ab|^{-1}$,
where $s=1$ in the indefinite case and $s=1/2$ in the definite case.

For $B=2|d(\mathcal{O})|^{1/2}\pi^{-1}s^{-1/2}$,
this volume equals
$ 4|ab|^{-1}|d(\O)|=16\mathrm{covol}(\M)$.
By Minkowski's convex body theorem, this implies that $V\cap \M$ contains
a non-zero element~$z = (u,v)\transpose$. We find $|\norm(z)| = |au\overline{u}+bv\overline{v}|\leq B$,
hence by choice of $b_1$ also $|a|\leq B$.

Finally, let $$ W = \{ (u,v)\transpose\in \C^2 : |\Re(u)|<\frac{1}{|a|}, |\Im(u)|<\frac{|d(\O)|^{1/2}}{2|a|}, |b|v\overline{v} \leq C\},$$
so $$\mathrm{vol}(W) = 2\pi \frac{|d(\O)|^{1/2}C}{|a^2b|},$$
which for $C=2|d(\O)|^{1/2}|a|\pi^{-1}$ is equal to
$4|d(\O)|/|ab|= 16\mathrm{covol}(\M)$.
By Minkowski's convex body theorem, this implies that $V\cap \M$ contains
a non-zero element~$z = (u,v)\transpose$.
We find $u=0$, hence $\Phi(b_1,z)=bv\not=0$
and $|\mu(z)|=|b|v\overline{v}\leq C$,
hence by choice of $b_2$, the number $|b|$ is below this bound, which proves
part~(\ref{itm:boundsforthatbasis1}).
\end{proof}

\subsubsection*{If $\mu$ takes the value 0}
Let $b_1\in \M$ with $\mu(b_1)=0$ be such that 
$zb_1\in \M$ for $z\in K$ implies $\abs{z}\geq 1$.
Then take $b_2\in \M$ such that
$\alpha=\Phi(b_1,b_2)\in \O$ has minimal absolute value
among all $b_2$ linearly independent of $b_1$,
and such that $b=\mu(b_2) \in\Z$ has minimal absolute value among all 
$b_2$ for the given value of~$\alpha$.

Identify $\M\otimes \Q$ with $K^2$ via the $K$-basis $b_1$, $b_2$.
\begin{lemma}\label{mu0}
 We have
\begin{enumerate}
 \item $\M\supset \O\times \O$ with index $N(\alpha)$,
 \item $1\leq N(\alpha)\leq 4\pi^{-2}|d(\O)|\quad\mbox{and}\quad
\abs{b}\leq |\alpha|\cdot  |d(\O)|^{1/2}$.
\end{enumerate}
\end{lemma}
\begin{proof}
  As the basis elements $b_1=(1,0)$ and $b_2=(0,1)$ are in $\M$,
  so is $\O\times \O$. By Lemma~\ref{lem:detisone},
  the index is $\pm \det S$. As we have $-\det S= N(\alpha)>0$,
  this proves~(1).

  Let $V = \{(u,v)\in \C^2 : N(u), N(v) < 1\}$, which has volume
  $\pi^2$ and note that $\M\subset \C^2$ has covolume
  $$\frac{|d(\O)|}{4[\M:\O\times\O]}=\frac{|d(\O)|}{4N(\alpha)}.$$
  By Minkowski's convex body theorem, if $\pi^2>4|d(\O)|/N(\alpha)$, then
  there exists a non-zero $z = ub_1+vb_2\in \M$ with $N(u)$, $N(v)<1$.
  Here $v=0$ would contradict minimality of~$b_1$
  (as $N(u)<1$),
  hence~$z$ is linearly independent of~$b_1$.
  Now $0<N(v)<1$ contradicts minimality of~$\alpha$.
  This proves $\pi^2\leq 4|d(\O)|/N(\alpha)$,
  which gives our upper bound on $N(\alpha)$.
  
  By translation of $b_2$ via an $\O$-multiple $\gamma b_1$
  of~$b_1$, we see
  $\abs{b}\leq \frac{1}{2}\abs{\Tr_{K/\Q}(\alpha\gamma)}$
  whenever the trace of~$\alpha\gamma$ is non-zero.
  To get~(2), we apply this to $\gamma=1$ or $\gamma=x$
  with $\O=\Z[x]$.
\end{proof}

We summarize the results in this section with the following approach to
find all isomorphism classes of pairs~$(\M,T)$. Lemmas~\ref{mu0}
and~\ref{lem:boundsforthatbasis} give a finite set of possibilities
for $a,b$ or $\alpha,b$ in terms of the discriminant of~$\Z[x]$.
For each, there is a finite set of possibilities for~$\M$,
as the index
$[\M : \Z[x] \times \Z[x]]$
is given in the appropriate lemma.
In particular, we can list all
possibilities for~$\M$. The matrix~$T$ then equals
\[\delta^{-1}\left(\begin{array}{cc} a & 0 \\ 0 & b\end{array}\right)
\quad\mbox{or}\quad
\delta^{-1}\left(\begin{array}{cc} 0 & \alpha \\
\overline{\alpha} & b\end{array}\right)\]
where $\delta = 2x-\Tr(x)$ generates the different ideal of~$\Z[x]$.
%The
%first matrix corresponds to Lemma~\ref{lem:boundsforthatbasis}, and the
%second matrix corresponds to Lemma~\ref{mu0}.

\subsection{Degree-4 CM-fields}\label{sec:degree4}
Let $K$ be a degree-$4$ CM-field, and fix an order $\O \subset K$ with
$\overline\O = \O$ that contains the order~$\Z[x]$.

We want to enumerate the isomorphism classes of 
pairs $(\mathcal{M},T)$, where
$\M\subset K$ is an $\mathcal{O}$-module of rank $4$ over $\Z$
and $T=(\xi)$ is a $1\times 1$-matrix such that 
\begin{enumerate}
\item \label{itm:totneg} $\xi^2$ is totally negative;
\item \label{itm:detone} the bilinear form $E:\mathcal{M}\times\mathcal{M}\rightarrow\Q$ given 
      by $E(x,y)=\Tr_{K/\Q}(\xi\overline{x}y)$ has image $\Z$ and determinant~$1$.
\end{enumerate}
Recall that we call $(\mathcal{M}_1,\xi_1)$ isomorphic to $(\mathcal{M}_2,\xi_2)$
if and only if there exists $u\in K^*$ satisfying
$u^{-1}\mathcal{M}_1=\mathcal{M}_2$ and $u\overline{u}\xi_1=\xi_2$. 

The fact that the image of the bilinear form is in $\Z$ is equivalent
to the inclusion of $\mathcal{O}$-ideals
$\xi\overline{\mathcal{M}}\subseteq \mathcal{M}^\dagger$, where
$\mathcal{M}^\dagger$ is the dual of $\mathcal{M}$ for the trace
form~$(x,y)\mapsto \Tr_{K/\Q}(xy)$.
The determinant of $E$ is $1$ if and only if this inclusion is an equality,
so we need to enumerate pairs $(\mathcal{M},\xi)$ satisfying
$\xi\overline{\mathcal{M}}= \mathcal{M}^dagger$,
up to isomorphism. (This observation is a generalization of
Theorems 3.13 and 3.14 of~\cite{spallek}.)
For any fractional $\mathcal{O}$-ideal $\mathcal{M}$, 
write $\mathcal{M}_K=\mathcal{M}\otimes_{\mathcal{O}}\mathcal{O}_K$.
If $e$ denotes the exponent of the abelian group $\mathcal{O}_K/\mathcal{O}$,
then for every $\mathcal{M}$, we have
$e\mathcal{M}_K\subseteq \mathcal{M}\subseteq \mathcal{M}_K$.
As a first step, we compute a complete set $C$ of representatives
of the class group of the maximal order $\mathcal{O}_K$ of~$K$.
For any $\mathcal{N}\in C$, it is a finite computation to find all $\mathcal{M}$
with $\mathcal{M}_K=\mathcal{N}$. 
This gives us a complete set of representatives of the ideals $\mathcal{M}$
up to multiplication by~$K^*$.

The next step is to find all possibilities for $\xi$ for each $\mathcal{M}$.
Let 
$
\mathcal{O}(\mathcal{M})=\{x\in K : x\mathcal{M}\subseteq \mathcal{M}\}
$ 
be the multiplier ring of~$\mathcal{M}$
and let $K_0$ be the real quadratic subfield of $K$.
For every pair $(\mathcal{M},\xi)$, we have
$$
\begin{array}{rrrrrrrrrr}
  \xi\overline{\mathcal{M}_K}&\supseteq &\xi\overline{\mathcal{M}}&=&\mathcal{M}^\dagger&\supseteq &
        \mathcal{M}_K^\dagger&=&\mathcal{M}_K^{-1}\mathcal{D}_{K/\Q}^{-1} &\mbox{and}\\
e\xi\overline{\mathcal{M}_K}&\subseteq& \xi\overline{\mathcal{M}}&=&\mathcal{M}^\dagger& \subseteq      &
         e^{-1}\mathcal{M}_K^\dagger&=&e^{-1}\mathcal{M}_K^{-1}\mathcal{D}_{K/\Q}^{-1}\rlap{,}
\end{array}
$$
so for every valuation $v$ of $K$, we have
$$ -2v(e)\leq  v(\xi)+v(\mathcal{M}_K\overline{\mathcal{M}_K}\mathcal{D}_{K/\Q})\leq 0.$$
This gives finitely many candidate ideals $\xi\mathcal{O}_K$.
For each,
we can check whether it is principal and
generated by an element $\xi\in K$ such that $\xi^2$ is totally negative.
If so, let $U$ be a complete set of representatives of those classes $\overline{u}$
in the finite 
group $\mathcal{O}_{K_0}^*/N_{K/K_0}(\mathcal{O}(\mathcal{M})^*)$
for which $(\mathcal{M},u\xi)$ satisfies (\ref{itm:totneg}) and (\ref{itm:detone}) above.
The pairs $(\mathcal{M},u\xi)$ cover all isomorphism classes.

\section{Finding models}\label{sec:modpol}
In this section we explain a method to find models for the subvarieties
corresponding to pairs~$(\M,T)$. First we pick coordinates for our
moduli space.
By Torelli's theorem~\cite[Theorem 11.1.7]{birkenhakelange},
the variety $\A_2$ contains the moduli space
of smooth curves of genus two as an open subvariety.
We let $I_2$, $I_4$, $I_6$, $I_{10}$ be the \emph{Igusa-Clebsch invariants},
denoted $A$, $B$, $C$, $D$ by Igusa~\cite{igusa}.
Igusa shows that this gives an injection
$\M_2(\C)\rightarrow \mathbf{P}^{(2,4,6,10)}(\C)$
into weighted projective space, and its image is the $(I_{10}\not=0)$-locus.
%which is an affine variety of dimension $3$
%and can be embedded into $\A^{10}$ by $10$ absolute Igusa invariants
%(i.e., homogeneous quotients of the $I_k$ of degree $0$).
We can describe points in $\M_2\subset\A_2$ either by smooth
curves of genus two, or by the values of the~$I_k$.

By Weil's theorem~\cite[Corollary 11.8.2]{birkenhakelange},
the complement $\A_2\setminus\M_2$ consists
only of points corresponding to products of elliptic curves with the product
polarization. We will therefore describe points in $\A_2\setminus \M_2$
by products of elliptic curves.

\subsection{Surfaces in the moduli space}
By Proposition~\ref{prop:humirr}, the space $\moduli$ contains exactly two
irreducible Humbert surfaces for $l \equiv 1 \bmod 4$ and exactly one
irreducible Humbert surface otherwise. Due to the size of the resulting
models, Humbert surfaces can only be computed for small discriminants. As
an example, the
surface of discriminant~$8$ that occurs for $l=2$ is computed 
in~\cite{gruenewalddata}. 
%        The defining equations on the web page are not
%        in terms of Igusa invariants, but what we say here
%        is true, but we used a version of the polynomial
%        that cannot be found directly on that web page.

\subsection{Curves in the moduli space}
By Theorem~\ref{thm:modulispaces} we obtain a Shimura curve if the field
$K = \Q(x)$ is imaginary quadratic and, using the notation from the
theorem, if we have $\det(T)>0$. Using the method from Section~\ref{ssec:imagquad}
we find all isomorphism classes of pairs~$(\M,T)$. Fix such a pair. 

The matrix~$S = \delta T$ is by assumption indefinite, and 
we give $B=\M\otimes \Q$ a quaternion algebra structure induced by $S$ as
defined by e.g.~\cite{shimura-unitary}.
This quaternion algebra structure is easiest described
by choosing a $K$-basis of $B$ that diagonalizes $S$
as
$$\left(\begin{array}{cc} 1 & 0 \\ 0 & s\end{array}\right).$$
Denote the first basis element of $B$ by $1$ and the second by~$\rho$,
so elements of $B$ are written as $x+y\rho$ with $x,y\in K$.
Then $B$ becomes a quaternion algebra with $\rho^2 = -s > 0$,
$\rho u=\overline{u}\rho$ for $u\in K$.
The quaternion algebra $B$ comes equipped with a positive definite 
involution $\dagger$ given by $(x+y\rho)^\dagger = \overline{x}+ y\rho$.
%% For our own reference:
%\alpha^\dagger = \delta^{-1} \overline \alpha \delta
%% \overline{u+vw} = \overline{u} + \overline{w} \overline{v}
%%                 = \overline{u} - w \overline{v}
%%                 = \overline{u} - v w
%% (u+vw)^\dagger  = u^\dagger + w^\dagger v^\dagger
%%                 = \overline{u} + \delta^{-1} \overline{w} \delta \overline{v}
%%                 = \overline{u} - \delta^{-1} w \delta \overline{v}
%%                 = \overline{u} - \delta^{-1} \overline{\delta} v w
%%                 = \overline{u} + v w

The quaternion algebra $B$ acts on $\mathcal{M}\otimes \Q$
(that is, on itself) by left multiplication,
and Shimura showed that
this induces an embedding $\iota:B\rightarrow \End(A)\otimes \Q$.
Moreover, the Rosati-involution on $\End(A)$ induces $\dagger$ on~$B$,
and $R = \iota^{-1}(\End(A))$ consists exactly of those elements of $B$
that map $\M$ to itself.

If we now take a maximal real quadratic subring $R_1\subset R$
on which $\dagger$ is trivial, then the Shimura curve corresponding
to $(\M, S)$ is contained in the Humbert surface of $R_1$.
If we take two such subrings $R_1$ and $R_2$ with non-isomorphic
fraction fields, then the intersection of the corresponding
Humbert surfaces is a finite union of irreducible curves,
and one of these
curves is the Shimura curve corresponding to~$(\M,S)$.

\subsection{CM-points for quartic fields}\label{sec:quarticmodel}
Section~\ref{sec:degree4} gives an algorithm to compute all pairs $(\M,T)=
(\M,(\xi))$ in case $\Q(x)$ is a degree four CM-field. The points corresponding
to $(M,(\xi))$ are abelian surfaces that have complex multiplication by~$(\M,(\xi))$.
Computing these abelian surfaces is a reasonably well studied problem.
From a pair $(M,(\xi))$, one obtains a complex
analytic abelian surface,
whose \emph{absolute Igusa invariants}
$I_2^5/I_{10}$, $I_2^3I_4/I_{10}$, $I_2^2I_6/I_{10}$
are algebraic numbers
and can be evaluated numerically.
With sufficient precision, one can make a guess for the exact algebraic
value, which is completely reliable in practice.
Unfortunately, the output of such algorithms is
currently proven to be correct
only in case~$\Q(x)$ contains no imaginary quadratic
subfield and $\Z[x]$ is the maximal 
order~\cite{goren-lauter,streng}.

\subsection{Other points in the moduli space}
Besides CM-points, there are two ways that we can get a point inside the 
moduli space. Indeed, if $x$ denotes an $(l,l)$-endomorphism, then we find
a point inside the moduli space if either $\Q[x]$ is not a field, or if
$\Q(x)$ is imaginary quadratic and, using the notation from 
Theorem~\ref{thm:modulispaces}, we have $\det(T)<0$. 

\subsubsection{If $\Q[x]$ is \emph{not} a field,}
Lemma~\ref{nonfieldlemma} gives 
restrictions on~$x$. Indeed, using
the notation of this lemma, we note that there are only finitely many pairs
$(\beta_1,\beta_2)$ of imaginary quadratic elements of complex absolute 
value~$\sqrt{l}$. For each pair, there are finitely many isomorphism classes 
of complex elliptic curves~$E_i$ with $\End(E_i) \supseteq \Z[\beta_i]$. 
Finally, for each pair of curves $(E_1,E_2)$ there are finitely many 
possibilities for the maximal isotropic
kernel~$\Gamma$ of~$\lambda$.
We can list all triples
$(E_1,E_2,\Gamma)$ to describe the points in the moduli space. 

\begin{remark} Not all~$\Gamma$ are allowed. Indeed, we only pick those
$\Gamma$ for which $x \in \End(E_1 \times E_2)$ induces an~$(l,l)$-endomorphism
on the quotient $(E_1\times E_2)/\Gamma$. This happens precisely
when $x(\Gamma) \subseteq \Gamma$.
\end{remark}

In many cases, we can obtain a model of $(E_1\times E_2)/\Gamma$
by \emph{glueing} the elliptic curves along~$\Gamma$. We
recall the general setup in the remainder of this subsection, and refer
to Section~\ref{explicitglue} for examples.

Let $E_1$ and $E_2$ be complex elliptic curves and let $n \in \Z_{>1 }$ be an
integer. Let $\psi: E_1[n] \isar E_2[n]$ be
an \emph{anti-isometry}
with respect to the Weil pairing,
i.e., an isomorphism such that
 $e_n(\psi(P),\psi(Q))=e_n(P,Q)^{-1}$ holds
for all $P,Q\in E_1[n]$.
If $\psi$ is an anti-isometry, then
$\graph\psi$ is maximally isotropic with respect
to the Weil pairing on $(E_1\times E_2)[n]$.
By Lemma~\ref{isotropicWeil},
this implies that $\graph\psi$ is the
kernel of an $(n,n)$-isogeny to a unique
principally polarized abelian variety
$A = (E_1 \times E_2) / \graph\psi$.
We say that we \emph{glue} $E_1$ and $E_2$ along 
their $n$-torsion via~$\psi$
and call the $(n,n)$-isogeny $E_1\times E_2\rightarrow A$
an $n$-glueing.
The following shows that glueings
are a very general kind of $(n,n)$-isogenies.
\begin{lemma}\label{lem:glueingseverywhere}
  Let $E_1,E_2$ be elliptic curves and $n$ a prime number.
  Then every $(n,n)$-isogeny $E_1\times E_2\rightarrow A$
  for any $A$ is either an $n$-glueing or of the form
  $E_1\times E_2\rightarrow (E_1/C_1)\times (E_2/C_2)$
  for $C_i\subset E_i$ of order~$n$.
\end{lemma}
\begin{proof}
  It suffices to prove that every isotropic subgroup 
  $\Gamma\subset E_1\times E_2$
  with $\Gamma\cong (\Z/n\Z)^2$ is either
  of the form $C_1\times C_2$ or is the graph of an anti-isometry
  $E_1\rightarrow E_2$.
  
  There are $(n^4-1)/(n-1)$ maximally isotropic subgroups by 
  Lemma~\ref{isotropicWeil},
  of which $$((n^2-1)/(n-1))^2=n^2+2n+1$$ are products
  $C_1\times C_2$ 
  and $(n^2-1)n=n^3-n$ are graphs of anti-isometries.
  As we have $(n^4-1)/(n-1)=(n^2+2n+1)+(n^3-n)$,
  this covers all cases.
\end{proof}

Some glueings are the Jacobian of a genus-$2$ curve.
In fact, we have the following result,
which we will not need in any of our proofs, but will
use as motivation later.
\begin{theorem}[{Kani~\cite[Thm.~3]{Kani}}]\label{thm:whenglueing}
Let $(E_1,E_2,n,\psi)$ be as above Lemma~\ref{lem:glueingseverywhere}
and assume~$n$ is prime.
Then the glueing 
associated to this pair is the Jacobian of a
genus-$2$ curve if and only if
there
do not
exist an integer $k$ strictly between $0$ and~$n$ and an isogeny
$h: E_1 \rightarrow E_2$ of degree $k(n-k)$ with the property
$$
h_{\mid E_1[n]} = k \circ f.
$$
\end{theorem}
In Section~\ref{sec:22endomorphisms} we will show how to write down explicit
equations for small~$n$. 

\subsubsection{If $\Q[x]$ \emph{is} a field,} the positive
definite case of Lemma~\ref{lem:boundsforthatbasis}
gives all pairs $(M,T)$.
From a pair $(M,T)$ we get a complex analytic abelian surface
to which we can (without proof) apply the same numerical techniques
as in Section~\ref{sec:quarticmodel}.
Alternatively, the pair $(M,T)$ relates the corresponding
abelian surface to an isogenous product of CM curves,
and we may hope to apply glueing techniques.

\section{Example}
\label{sec:22endomorphisms}

In this section we will solely focus on the case~$l=2$, and we will identify
all irreducible subvarieties of~$\modulitwo$ by giving explicit models.
We use the computer algebra package MAGMA~\cite{magma} to help us with the 
computations. For the points inside~$\modulitwo$ we either give a
hyperelliptic model $y^2=f(x)$, or
Igusa-Clebsch invariants $I_2$, $I_4$, $I_6$, $I_{10}$
(introduced as $A$, $B$, $C$, $D$ by Igusa~\cite{igusa},
who based them on invariants of Clebsch).
Mestre's algorithm~\cite{mestre,bouyer-streng} computes
a hyperelliptic model from its Igusa-Clebsch invariants.
%and an implementation is available for totally real
%base fields of degree~$\leq 2$.
%
% To keep in comments:
%        For the curve C_{-20}^i, there is no implementation
%        available of Mestre's algorithm
%        that gives a nice model, but we have
%        computed that curve, and actually all curves
%        except the simple ones and C_{-4, -7}, by
%        glueing anyway, so we do have a model over an
%        extension field, for what that is worth.
%        It is:
%        C((1+2sqrt(-sqrt5-2))/2, (1+2sqrt(sqrt5-2))/2),
%        where the
%        signs of the square root of 5 must be picked
%        consistent, and the other
%        square roots can be varied, giving two distinct curves.

We start by introducing short notation for the elliptic curves
and curves of genus~$2$ that appear in our lists.
If $D$ is a negative quadratic discriminant and $h(D)$ is the class number of 
the order $\mathcal{O}_D$ of discriminant~$D$, 
let $E_D^1,\ldots,E_D^{\smash{h(D)}}$ be the isomorphism classes of elliptic 
curves over $\C$ with endomorphism ring~$\mathcal{O}_D$.
For $h(D)=1$, we denote $E_D^1$ also by $E_D$.
Table~\ref{tableofcurves} on page~\pageref{tableofcurves}
defines a list of curves of genus~$2$ in terms
of their Igusa-Clebsch invariants.
All curves in the table are isogenous to a product of elliptic curves.
In fact, these curves have a subscript $D$ or $D_1,D_2$
in their notations, and these subscripts indicate
that their Jacobian 
is isogenous
to $E_D^1\times E_D^1$ or $E_{D_1}^1\times E_{D_2}^1$.
\begin{table}
$$
\begin{array}{|l|l|}\hline
   C & (I_2:I_4:I_6:I_{10})(C)\\
\hline
  C_{-3} & (40:45:555:6)\\
% (lies on the Humbert surface, absolute invariants
% [ 51200000/3, 480000, 148000 ])
C_{-6}^1 & ( 92: 108: 4104: 24)\\
% (lies on the Humbert surface, absolute invariants
% [ 823851904/3, 3504096, 1447344 ])
C_{-6}^2 & (76: 252: 5160: 24)\\
% (lies on the Humbert surface, absolute invariants
% [ 316940672/3, 4609248, 1241840 ])
C_{-7} & (10840: 2004345: 7846230105:131736761856)\\
% (lies on the Humbert surface, absolute invariants
% [ 292332062070200000/257298363, 1256346006875/64827, 10888794125875/1555848 ])
C_{-8} & ( 20: -20: -40: 8)\\
% (lies on the Humbert surface, absolute invariants
% [ 400000, -20000, -2000 ])
  C_{-15} & (20: 225: 1185: -384)\\
% (lies on V OUTSIDE the Humbert surface, absolute invariants
% [ -25000/3, -9375/2, -9875/8 ])
  C_{-20}^i &  
(46i + 42: 210i - 220: 50i - 4350: -73i + 161)
\quad   (i^2=-1)\\
% Lies on the Humbert surface
  C_{-4,-7}^a &  ( 8+20a:  - 1035-450a:   87246+33606a: 
      25164+9504a)\\ & \quad (a^2=7)\\
% (lies on V OUTSIDE the Humbert surface, absolute invariants
% [ 1/2187*(4029826304*a - 10625457664), 1/27*(3779440*a - 10269440), 
%   1/81*(2940280*a - 6726824) ])
  C_{-4,-8} & (24: 30: 366: 2)\\
% (lies on V OUTSIDE the Humbert surface, absolute invariants
% [ 3981312, 207360, 105408 ])
\hline\end{array}
$$
\caption{A list of genus-$2$ curves and the
Igusa-Clebsch invariants that define them.
}
\label{tableofcurves}
\end{table}

\subsection{Explicit glueings}\label{explicitglue}
To compute models for all irreducible subvarieties of~$\modulitwo$, we need
to explictly compute $(2,2)$-isogenies. Legendre first discovered a way
of doing this. For the convenience of the reader, we recall 
Serre's explanation~\cite[Sec.\ 27]{Ser85} of Legendre's method. 

Let $(E_1,E_2,\psi)$ be a triple consisting of two elliptic curves, and
an anti-isometry $\psi: E_1[2] \rightarrow E_2[n]$.
We choose a Weierstra{\ss} model $y^2=g_i$ for~$E_i$.
The double cover $E_1\rightarrow\mathbf{P}^1$
given by the $x$-coordinate has four ramification points,
which are exactly
the points in $E_1[2]=\{P_1,P_2,P_3,P_4\}$.
Let $Q_j=\psi(P_j)$ and apply a fractional linear transformation
of the $x$-coordinate of $E_2$ to get $x(Q_j)=x(P_j)$ for $j=1,2,3$.
\begin{remark}\label{rem:whenglueing}
By Theorem~\ref{thm:whenglueing}, the glueing
corresponding to $(E_1,E_2,\psi)$ as above
is
the Jacobian of a curve of genus~$2$ if and only
if $\psi$ is not of the form $h_{|E_1[2]}$ for an isomorphism
$h:E_1\rightarrow E_2$, that is, if and
only if $x(Q_4)$ and $x(P_4)$ are distinct.
\end{remark}
By Remark~\ref{rem:whenglueing}, it makes sense to
restrict to the case where
$x(Q_4)\not=x(P_4)$.
Let $K_{1}$ and $K_{2}$ be the function fields of $E_{1}$ and $E_{2}$
as quadratic extensions of $\Q(x)$ and let
$K_{12}$ be their composite. Let $K_{0}\subset K_{12}$
be the other intermediate extension
of degree~$2$. The field $K_{0}$ is ramified over $\Q(x)$ only at $x(P_4)$ and
$x(Q_4)$, so $K_0$ has genus~$0$.
The extension $K_{12}/K_0$ is ramified in the $6$ places of $K_0$
lying over the three points $x(P_j)$ for $j=1,2,3$.
In particular, it follows from the Riemann-Hurwitz
genus formula that the curve $C$ with function field $K_{12}$
has genus~$2$. 
The Jacobian of $C$ is
$(2,2)$-isogenous to $E_{1}\times
E_{2}$ via an isogeny with kernel $\graph \psi$.

In terms of equations, Legendre's
glueing construction can be described as follows.
Choose a Legendre
model of $E_1$ so that we have $x(P_1)=\infty$, $x(P_2)=0$, $x(P_3)=1$
and let $a=x(P_4)$.
As above we have $x(Q_j)=x(P_j)$ for $j=1,2,3$.
Let  $b=x(Q_4)\not=a$, so
$E_1$ and $E_2$ are given by
\[
E_{1}:y^{2}=x(x-1)(x-a),\quad
E_{2}:y^{2}=x(x-1)(x-b).
\]
The curve $C_0$ with function field $K_0$ is given by $C_0:y^2=(x-a)(x-b)$
and has a parametrization \[t\mapsto \left(x=a+\frac{(a-b)t^2}{1-t^2},\quad y=\frac{(a-b)t}{1-t^2}\right).\]
The extension $K_{12}/K_0$ ramifies at the points
with $x\in\{\infty,0,1\}$,
i.e., with $t^2\in\{1,a/b,(a-1)/(b-1)\}$.
In particular, the 
genus-$2$ curve $C=C(a,b)$ obtained by glueing $E_{1}$ and
$E_{2}$ as above is:
\[
C(a,b):s^{2}=(t^2-1)(t^2-\frac{a}{b})(t^2-\frac{a-1}{b-1}).
\]
%% More details, which we did not need.
%The map from $C$ to the $x$-line is given by
%$(t,s)\mapsto x$ with $x$ as in the parametrization above.
%From this, one derives that the map $C\rightarrow E_1$
%sends $(t,s)$ to $(x,y)$ with $x$ as above
%and $y=\sqrt{x(x-1)(x-a)}\in \Q(t,s)$.
%% We're not sure about the following formula, but have not
%% used it.
%\[y=\sqrt{x(x-1)(x-a)}=\sqrt{b(b-1)(b-a)}\frac{st}{(t^2-1)^4}.\]

\subsection{Real quadratic field}
The only real quadratic field occuring for $l=2$ is $\Q(\sqrt{2})$ and the
only order is $\Z[x]=\Z[\sqrt{2}]$.
By Lemma~\ref{lem:humbertsurface},
there is only a single Humbert surface $H_{8}$,
for which defining equations have been
computed by
Gruenewald~\cite{gruenewald,gruenewalddata}.
This deals with all the surfaces inside~$\modulitwo$.

\subsection{Imaginary quadratic fields, definite~$S$}
Now suppose $\Z[x]$ is imaginary quadratic.
We start with the case where the Hermitian form $S$
of Section~\ref{ssec:imagquad} is definite.
\begin{lemma}\label{lem:case3}
  Let $A$ be a principally polarized abelian surface over $\C$
  and $x:A\rightarrow A$ a $(2,2)$-isogeny.
  If $\Z[x]$ is imaginary quadratic and $S$ of Section~\ref{ssec:imagquad}
  is
  definite,
  then $A$ is isomorphic to one of
  $E_{-4}\times E_{-4}, E_{-7}\times E_{-7}, E_{-8}\times E_{-8}$, and $J(C_{-8})$.
\end{lemma}
\begin{proof}
  Example \ref{quadraticfields} shows that $\Z[x]$ is the maximal
  order of the imaginary quadratic field of discriminant $D\in \{-4,-7,-8\}$.
  By replacing $(M,T)$ by its complex conjugate, we may assume 
  that $S = \delta(\Z[x])T$ is \emph{positive} definite.

In all three cases, the order~$\Z[x]$ has class number one, and the only
rank two projective $\Z[x]$-submodule of $K\times K$ equals
$$
\M = \Z[x] \times \Z[x]
$$
up to $\GL_2(\Q(x))$-equivalence. Furthermore, in all three cases one
equivalence class of $S$ is given by
$$
S = \left(\begin{array}{cc} 1 & 0 \\ 0 & 1\end{array}\right),
$$
and we obtain the three products $E_D \times E_D$ with the product 
polarizations. 

In this situation, it is easier to use the techniques 
of~\cite[\S 5]{hayashidanishi} instead of the method
from Section~\ref{ssec:imagquad}
to find all isomorphism classes of
pairs $(\M,S)$.
In any case, we find that these products $E_D\times E_D$
are all classes for $D\in\{-4,-7\}$. For $D=-8$ we
find that there is exactly one further class, given by
$$
S = \left(\begin{array}{cc} 2 & \sqrt{-2}+1 \\ -\sqrt{-2}+1 &  2\end{array} \right),
$$
which corresponds to giving the product $E_{-8}\times E_{-8}$ a different
polarization.
We will see that it also corresponds to a $(2,2)$-glueing.

To use Legendre's
method, we first choose a map
$$
\psi : E_{-8}[2] \rightarrow E_{-8}[2]
$$
to be used for the glueing. Recall that $x = \sqrt{-2}$
is a $(2,2)$-endomorphism
on $E_{-8}\times E_{-8}$. If we let $g : E_{-8} \times E_{-8} \rightarrow
A = E_{-8}\times E_{-8}/(\graph \psi)$ denote the quotient map, then we require
that $gxg^{-1} \in \End(A) \otimes\Q$ is in fact an endomorphism of~$A$.
Indeed, if this is the case, then $gxg^{-1}$ is a~$(2,2)$-isogeny of~$A$.

We pick the basis 
 $\frac{1}{2},\frac{1}{2}\sqrt{-2}\pmod{\Z[\sqrt{-2}]}$ of $E_{-8}[2]$,
where $E_{-8} =\C/\Z[\sqrt{-2}]$;
and we define 
$$
\psi=\left(\begin{array}{cc} 1 & 0 \\ 1 & 1\end{array}\right).
$$
The equality $x=\sqrt{-2}=({0 \atop 1} {0\atop 0})$
shows that $\psi$ and $x$ commute on $E_{-8}[2]$. In particular,
 multiplication by $x$ maps $\graph\psi$ into itself, and $gxg^{-1}$ is
an endomorphism of~$A$. 
Using the formula of Section~\ref{explicitglue},
  we find $A=J(C_{-8})$ with $C_{-8}$ as in
  Table~\ref{tableofcurves}.
In particular, $A\not\cong E_{-8}\times E_{-8}$
and we have found the final abelian surface.
\end{proof}

\subsection{Imaginary quadratic fields, indefinite~$S$}

\begin{lemma}\label{lem:case2}
  Let $A$ be a principally polarized abelian surface over~$\C$.
  There exists a $(2,2)$-isogeny $x:A\rightarrow A$ 
  with $\Z[x]$ imaginary quadratic and $S$ \emph{in}definite,
  if and only if $A$ is
  a point on one of the following Shimura curves:
\begin{enumerate}
 \item the curve in $\A_2$ of which the points are
       the products $E\times E$ of an elliptic curve $E$ with itself;
 \item the curve in $\A_2$ of which the points are
       the products $E\times F$ for all $2$-isogenies $E\rightarrow F$;
\item  the Zariski closure
       of the image of the map
       $\C\setminus\{0,1\}\rightarrow \A_2$
       given by $u\mapsto C(u,1-u)$,
       where $C(a,b)$ is given as in Section~\ref{explicitglue}
       by
    $$C(a,b): y^2=(x^2-1)(x^2-\frac{b}{a})(x^2-\frac{b-1}{a-1});$$
\item  the Zariski closure
       of the image of the map
       $\C\setminus\{0,\pm 1, \pm i, \pm \sqrt{2}-1\}\rightarrow \A_2$
       given by $$u\mapsto C(u^2,\frac{(u-1)^2}{(u+1)^2});$$
% (checked and indeed lies on H_8)
\item  the Zariski closure
       of the image of the rational
       map $X\dashrightarrow\A_2:(u,v)\mapsto C(u,v)$,
       where $X\subset \mathbf{A}^2$ is given by
       Figure \ref{fig:varietyX} on page~\pageref{fig:varietyX}.   
\end{enumerate}
Each of these $5$ Shimura curves has genus~$0$.
\end{lemma}
\begin{figure}
\fbox{\parbox{\textwidth}{
  The curve $X\subset \mathbf{A}^2$ given by
  \begin{eqnarray*} X &:& a,b\not\in\{0,1\},\\
& & 0=a^8 - 16777216a^7b^7 + 58720256a^7b^6 - 78905344a^7b^5\\
 & &+ 50462720a^7b^4 - 15307264a^7b^3 + 1858304a^7b^2 - 51464a^7b\\
& & + 58720256a^6b^7 - 499580928a^6b^6 + 1158348800a^6b^5\\
& &- 916944896a^6b^4 +    63926016a^6b^3 + 133672476a^6b^2\\
& & + 1858304a^6b
 - 78905344a^5b^7
 +  1158348800a^5b^6\\ & & - 3908889600a^5b^5
+ 4686427648a^5b^4 - 1905600312a^5b^3\\ & & + 63926016a^5b^2
- 15307264a^5b + 50462720a^4b^7 - 916944896a^4b^6\\
& & + 4686427648a^4b^5 - 7639890874a^4b^4 + 
    4686427648a^4b^3\\
& & - 916944896a^4b^2 + 50462720a^4b - 15307264a^3b^7 +
    63926016a^3b^6\\
& & - 1905600312a^3b^5 + 4686427648a^3b^4 - 
    3908889600a^3b^3\\
& & + 1158348800a^3b^2 - 78905344a^3b + 1858304a^2b^7\\
& & +
    133672476a^2b^6 + 63926016a^2b^5 - 916944896a^2b^4\\
& & + 
    1158348800a^2b^3 - 499580928a^2b^2 + 58720256a^2b - 51464ab^7\\
& & + 
    1858304ab^6 - 15307264ab^5 + 50462720ab^4\\
& & - 78905344ab^3 + 
    58720256ab^2 - 16777216ab + b^8,
\end{eqnarray*}
has genus~$2$. It is birational to the fibered product
$Y_0(7)\times_{Y(1)}Y(2)$ parametrizing
tuples $(E,F,\psi,P_1,P_2)$ up to isomorphism,
where $\psi:E\rightarrow F$ is a $7$-isogeny
and $P_1,P_2$ form a basis of $E[2]$.
The birational map $X\rightarrow Y_0(7)\times_{Y(1)}Y(2)$
is given by $(a,b)\mapsto (\psi,P_1,P_2)$,
where $E:y^2=x(x-1)(x-a)$, $P_1=(0,0)$,
$P_2=(1,0)$, 
$F:y^2=x(x-1)(x-b)$ and $f:E\rightarrow F$ is a $7$-isogeny.
%Moreover, we have $f(0,0)=(0,0)$, $f(1,0)=(1,0)$.
}}
\caption{The curve $X$ of Lemma \ref{lem:case2}.}\label{fig:varietyX}
\end{figure}
\begin{proof}
  As in the proof of Lemma~\ref{lem:case3},
  the ring $\Z[x]$ is the maximal
  order 
  of the imaginary quadratic field of discriminant $D=-4$, $-7$, or~$-8$.
  It is not hard to see that the following pairs $(\M, T)$
% Earlier versions had all details, so we can always dig those up.
  cover all isomorphism classes for $|D|\leq 8$: take
  $\mathcal{M}=\Z[x]\times \Z[x]$
  and
$$ 
T_1 = \frac{1}{\delta} \left( \begin{array}{cc} 1 & 0 \\ 0 & -1\end{array}\right)\quad
\mbox{or}\quad T_2 = \frac{1}{\delta}\left(\begin{array}{cc} 0 & 1 \\ 1 & 0\end{array}\right).
$$
  Note that $T_1$ and $T_2$ are equivalent for odd~$D$, so that we actually
  have at most 5 distinct Shimura curves, which we will now compute.
  Of course one also arrives at five cases if one
  uses Section~\ref{ssec:imagquad}.

  Let $f:E\rightarrow F$ be a degree-$n$ isogeny
  of elliptic curves
  and let $f^\vee$ be its dual.
  If  $a,b\in\Z$ are such that $a^2+b^2n=l$,
  then we have an $(l,l)$-isogeny
  $$x=\left(\begin{array}{cc} a & b f^\vee \\ -b f & a\end{array}\right)\in\End(E\times F)$$
  such that $\Z[x]$ is quadratic of discriminant $D=-4b^2n$.
  Indeed, we have $x^2-2ax+l=0$ and $x'=2a-x$, hence $x'x=l$.
  
  The products $E\times F$ form a curve $\mathcal{S}$
  in the moduli space $\A_2$, and that curve $\mathcal{S}$
  is covered by the modular curve $Y_0(n)$
  of $n$-isogenies $E\rightarrow F$.
  Any such $\mathcal{S}$ is contained in one
  of the Shimura curves we are looking for.
  By the construction in~\cite[\S 9.6]{birkenhakelange},
  the Shimura curves are irreducible, so that
  $\mathcal{S}$ is in fact equal to one of the Shimura curves.
  
  We apply the above to $l=2$, $D=-4$, $n=a=b=1$, $F=E$, $f=\mathrm{id}_E$,
  which yields the Shimura curve of all squares $E\times E$
  of elliptic curves $E$.
  We also apply it to
  $l=2$, $D=-8$, $n=2$, $a=0$, $b=1$,
  which yields the Shimura curve of all products
  $E\times F$ for which $F$ is $2$-isogenous to $F$.
  We have now found $2$ of our $5$ Shimura curves.

  We will find more Shimura curves by looking at
  $(k,k)$-isogenies $g:E\times F\rightarrow A$ for small $k$
  and checking if the element
  $gxg^{-1}\in\End_\Q(A)$ (with $x$ as above)
  is in $\End(A)$.
  As in the proof of Lemma~\ref{lem:case3},
  the fact that $g$ is a $(k,k)$-isogeny
  and $x$ is an $(l,l)$-isogeny imply
  that the induced endomorphism
  $gxg^{-1}\in\End(A)$
  is an $(l,l)$-isogeny.

  The condition that $gxg^{-1}$ is in $\End(A)$
  is equivalent to the condition that $x$ maps $\ker(g)$ to itself.
  In particular, as $\ker(g)$ is the graph of an anti-isometry
  $\psi:E[k]\rightarrow F[k]$, we have $gxg^{-1}\in\End(A)$
  if and only if
  \begin{equation}\label{eq:conditionquotient}
  b(\psi f^\vee\psi + f)=0
  \end{equation} holds on $E[k]$.

  We start with the case $D=-4$, and choose
  again $l=2$, $n=a=b=1$, $F=E$, $f=\mathrm{id}_E$.
  Write $E:y^2=x(x-1)(x-u)$
  and choose the basis $(0,0)$, $(1,0)$ of $E[2]$.
  We choose~$\psi$ to
  be the automorphism of $E[2]$ that swaps these basis elements,
  and it is easy to see that~\eqref{eq:conditionquotient} holds
  in this case.
  The fractional linear transformation of the $x$-coordinate
  sending $\infty$ to itself and swapping $0$ and $1$
  maps $E$ to $y^2=x(x-1)(x-v)$ with $v=1-u$.
  Therefore, we find $E\times E/(\graph\psi)\cong J(C(u,1-u))$
  as in Section~\ref{explicitglue}.
  This yields our third Shimura curve.
  
  For the case $D=-8$ (with $n=l=2$, $a=0$, $b=1$),
  recall from~\cite[Ex.~X.4.8]{silverman} that  
  for every $2$-isogeny $f:E\rightarrow F$,
  the pair of curves $E$ and~$F$
  can be written in 
  the form $$E:y^2=x(x-1)(x-a^2),\quad F:y^2=x(x-1)(x-(\frac{1-a}{1+a})^2)$$
  with $\ker f=\langle(0,0)\rangle=f^\vee(F[2])$,
  $\ker f^\vee=\langle(0,0)\rangle=f(E[2])$ and $a\in\C\setminus \{0,\pm 1\}$.
  The map $\psi:E[2]\rightarrow F[2]$ given by $(0,0)\mapsto (0,0)$ and $(1,0)\mapsto (1,0)$
  satisfies $\psi f^\vee\psi=f$ and we find Shimura curve~(4).

  For the only remaining Shimura curve (with $D=-7$), we apply the above to
  $l=8$, $D=-28$, $n=7$, $a=b=1$.
  Let $f:E\rightarrow F$ be a $7$-isogeny and
  let $\psi_{j}=f_{|E[j]}:E[j]\rightarrow F[j]$.
  Let $A=E\times F/(\graph \psi_2)$ and let $g$ be the quotient map.
  We will show that $\frac{1}{2}gxg^{-1}$ is an endomorphism
  of~$A$, i.e.~that $x$ maps $[2]^{-1}\graph\psi_2$ to $\graph\psi_2$.
  This is slightly more complicated than~\eqref{eq:conditionquotient},
  but the ideas are mostly the same.

  Choose a basis of~$E[4]$. As $f$ has odd degree, the
  image of this basis under $f$ is a basis of~$F[4]$.
  We have $2^{-1}\graph\psi_2=\graph\psi_4+(E[2]\times 0)$.
  Given $P\in E[4]$, we find
  $x(P,\psi_4(P))=(P+f^\vee\psi_4P,-fP+\psi_4P)=0$.
  Let $P\in E[2]$ be any element. We find
  $x(P,0)=(P,-fP)=(P,fP)\in\graph\psi_2$.
  This shows that indeed $x$ maps $[2]^{-1}\graph\psi_2$ to $\graph\psi_2$,
  hence~$A$ has a (2,2)-isogeny $\frac{1}{2}g xg^{-1}$
  to itself with $\Z[\frac{1}{2} gxg^{-1}]$ quadratic
  of discriminant~$-7$.
  This gives the fifth Shimura curve
  $\mathcal{S}\subset\A_2$, and we will now write down an explicit model
  for that, which is more complicated
  as we need a modular curve~$X$ that parametrizes
  $7$-isogenies $f:E\rightarrow F$ of elliptic curves
  with a basis of the $2$-torsion of~$E$.
  
  As before, let $Y_0(l)$ be the modular curve of
  $l$-isogenies $f:E\rightarrow F$.
  Let $Y(l)$ be the modular curve of elliptic curves $E$
  together with a basis of the $l$-torsion.
  We have natural maps $\pi_E,\pi_F:X\rightarrow Y(2)$,
  given by $\pi_E(x)=(E,P_1,P_2)$
  and $\pi_F(x)=(F,f(P_1),f(P_2))$, and natural maps
  $\pi_E,\pi_F:Y_0(l)\rightarrow Y(1)=\mathbf{A}^1$,
  given by $\pi_E:f\mapsto j(E)$ and $\pi_F:f\mapsto j(F)$.
  See also Figure \ref{fig:varietiesinclX} on page~\pageref{fig:varietiesinclX}.
  For ease of notation, we identify~$X$ (respectively~$Y_0(7)$)
  birationally
  with its image in $Y(2)\times Y(2)$ (respectively $Y(1)\times Y(1)$).
  Then $Y_0(l)$ is exactly the zero-locus of the modular polynomial
  $\Phi_l\in\Z[X,Y]$.

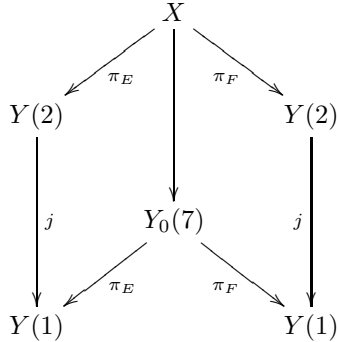
\begin{figure}
$$\xymatrix{      
         & X\ar[dl]^{\pi_E}\ar[dr]_{\pi_F}\ar[dd]\\
       Y(2)\ar[dd]^{j}        &                  &      Y(2)  \ar[dd]_{j}      \\
                   &         Y_0(7) \ar[dl]^{\pi_E}\ar[dr]_{\pi_F}    &          \\
       Y(1)        &                  &      Y(1)         }$$
\caption{The curves $X$, $Y(2)$, $Y(1)$ and $Y_0(7)$
from the proof of Lemma~\ref{lem:case2}.}\label{fig:varietiesinclX}
\end{figure}  

  Next, we have $Y(2)=\C\setminus\{0,1\}$,
  where $a\in \C\setminus\{0,1\}$
  corresponds to the elliptic curve $y^2=x(x-1)(x-a)$
  with the basis
  $(0,0), (1,0)$ of the $2$-torsion.
  To glue $E$ to $F$ along $\graph f_{|E[2]}$ for all $x\in X$, we need a description
  of $X$ from which $a=\pi_E(f)$ and $b=\pi_F(f)$ can be read off.
  This means that we want a description of $X$ as a subvariety of
  $Y(2)\times Y(2)$.
  We have the natural map $j:Y(2)\rightarrow Y(1)$ given by
  $j(a)=2^8(a^2-a+1)^3(a-1)^{-2}a^{-2}$. Let $W=(j\times j)^{-1}(Y_0(7))$,
  so $X\subset W\subset Y(2)\times Y(2)$ and $W$
  is the zero locus of $P=\Phi_7(j(u),j(v))\in\Q(u,v)$.
  
  We factor the numerator of $P$ using MAGMA
  as a product of $6$ irreducible polynomials in $\Q[u,v]$.
  Note $W=\sqcup_\sigma X^\sigma$, where $\sigma\in\Aut(F[2])$
  acts by changing the basis of $F[2]$.
  Now we only need to identify $X$ as
  one of the~$6$ factors of the numerator of~$P$.

  To select the correct factor, we now prove that $\mathcal{S}$ is a
  subset of the Humbert surface~$H_8$.
  Let
  $$y=\left(\begin{array}{cc} 1 & f^\vee \\ f & -1\end{array}\right),\quad\mbox{so}\quad
              \frac{1}{2}(x-y)=\left(\begin{array}{cc} 0 & 0 \\ -f & 1\end{array}\right).$$
  We compute that $g\frac{1}{2}(x-y)g^{-1}$ (and hence $g\frac{1}{2}yg^{-1}$)
  is an endomorphism of $A$.
  We have $y'=y$ and $(\frac{1}{2}y)^2=2$,
  hence $\mathcal{S}\subset H_8$.
  Using the defining equation for $H_8$
  from~\cite{gruenewalddata},
% We actually used one in terms of absolute Igusa invariants,
% supplied by Gruenewald via email, but the ones on his web page
% http://echidna.maths.usyd.edu.au/~davidg/ThesisData/RosenhainHumbert/H8/h8components.txt
% in terms of Rosenhain invariants work at least as well,
% since we have given the curve as C(a,b), and from a,b
% one gets the Rosenhain invariants.
  we find that exactly one out of our $6$ candidates for $S$ lies
  on~$H_8$, so we know that this is the correct one
  and we present it in Figure~\ref{fig:varietyX}.
  
  Now only the statement about the genus remains.
  The first two Shimura curves are covered by
  $Y(1)$ and~$Y_0(2)$, which are known
  to have genus~$0$. The third and fourth are
  defined as the image of a rational map from~$\mathbf{P}^1$, hence also
  have genus~$0$.
  We used MAGMA to compute that the geometric genus of
  the plane curve~$X$
  (which covers the fifth Shimura curve) is also~$0$.
\end{proof}

\subsection{Degree 4 CM-points}
We recall that a \emph{Weil $q$-number} is an algebraic integer $\pi$ such that $\pi \overline{\pi} = q$
for every embedding of $\pi$ into $\C$.
\begin{lemma}\label{lem:case4}
If $K=\Q[x]$ is a quartic CM-field and $x$ is a Weil $2$-number, then
the characteristic polynomial $P_x^{\mathrm{a}}$ is an irreducible
polynomial of degree~$4$ of the form
$$f=X^4 - a_1X^3 + (4+a_2)X^2 - 2a_1X + 4,$$
where $a_1$ and $a_2$ are integers.

Tables \ref{v4table} and \ref{primitivetable} list exactly the possible pairs $(a_1,a_2)$ (up to replacing $x$ by $-x$ and $a_1$ by $-a_1$)
such that $f$ is irreducible and its roots satisfy $x\overline{x}=2$,
together with all complex principally polarized abelian surfaces
of which the endomorphism ring contains a subring isomorphic to $\Z[X]/f$.
The Galois group of~$f$ is $C_2\times C_2$ for~$f$ as in Table \ref{v4table}
and $D_4$ for$f$ as in Table \ref{primitivetable}.\end{lemma}
\begin{table}
$$
\begin{array}{|l|l|l|}\hline
  a_1 & a_2 & \mbox{principally polarized abelian varieties}\\
\hline
0 & -7 & E_{-7}\times E_{-7}\\
0 & -6 & J(C_{-8}), E_{-3}\times E_{-12}\\
0 & -5 & J(C_{-3}), E_{-3}\times E_{-3}, J(C_{-15}), E_{-15}^1\times E_{-15}^2\\
1 & -5 & E_{-3}\times E_{-3}\\
0 & -3 & J(C_{-20}^{ i}) (i^2=-1), J(C_{-15}), E_{-15}^1\times E_{-15}^2 \\
0 & -2 & J(C_{-6}^1), J(C_{-6}^2), E_{-3}\times E_{-3}, E_{-3}\times E_{-12}\\
2 & -2 & E_{-3}\times E_{-3},  E_{-4}\times E_{-4}\\
3 & 1 & J(C_{-15}), E_{-3}\times E_{-3}\\
\hline\end{array}
$$
\caption{All Weil $2$-numbers that generate quartic fields that contain an
imaginary quadratic subfield
are roots of $f=X^4 - a_1X^3 + (4+a_2)X^2 - 2a_1X + 4$ with $a_1,a_2$ as in the table. For each such Weil $2$-number $x$, the column
on the right lists
all principally polarized abelian surfaces
over $\C$
of which the endomorphism ring contains~$x$.}\label{v4table}
\end{table}

\begin{table}
$$\begin{array}{|l|l|l|}\hline
a_1 & a_2 & [D,A,B] \\
\hline
1 & -4 & [17,5,2]\\
% (lies on V OUTSIDE the Humbert surface, absolute invariants
% [ 30233088*a - 45349632, 524880, -8748*a + 139968 ])
1 & -3 &  [13,9,17]\\
% 1 & -3 &  [-18*o + 6,  -375*o - 540,  2031*o + 3636,  176*o + 272], o^2-o-4=0
% (lies on V OUTSIDE the Humbert surface, absolute invariants
% [ 1958337*a - 8218017, 1/4*(374625*a - 1008855), 1/16*(408213*a - 1368387) ])
1 & -1 &  [5,13,41]\\
% (lies on V OUTSIDE the Humbert surface, absolute invariants
% [ 18128043*a - 115322697, 1/4*(1489347*a - 9073863), 1/16*(2276667*a - 14410143)])
2 & -1 & [8,10,17]\\
% (lies on the Humbert surface AND with additional multiplicity on V,
% absolute invariants
% [ 236196*a + 2125764, 1/4*(164025*a + 426465), 1/16*(234009*a + 216513) ])
\hline\end{array}
$$
\caption{All Weil $2$-numbers that generate quartic fields
that do not contain an imaginary quadratic subfield are roots of 
$f=X^4 - a_1X^3 + (4+a_2)X^2 - 2a_1X + 4$ with $a_1,a_2$ as in the table.
For each such Weil $2$-number~$x$,
the ring $\Z[x,\overline{x}]$ turns out to be the maximal
order~$\O_K$ of its field of fractions $K\cong\Q[X]/(X^4+AX^2+B)$,
with real quadratic subfield $K_0$ of discriminant~$D$.
For each, there are~$2$ isomorphism classes of principally
polarized abelian surfaces with endomorphism ring~$\O_K$,
and they are Jacobians of genus-two curves.
Conjecturally correct Igusa-Clebsch invariants
% actually absolute Igusa invariants, but that is equivalent
(obtained by numerical
evaluation) of these eight curves
can be found by searching for $[D,A,B]$
in the Echidna database~\cite{echidna-data}.
Hyperelliptic models corresponding to these invariants,
and in some cases even a proof of correctness,
appear
in~\cite{bouyer-streng}.
In general, one could prove correctness using the methods
of~\cite{wamelen-proving} or the results
of~\cite{lauter-viray}.
}\label{primitivetable}
\end{table}

\begin{proof} 
Suppose $x$ is a Weil $2$-number of degree~$4$ and let $\beta=x+2/x=x+\overline{x}$, so $\beta$ is a
totally real algebraic integer of degree at most $2$.
As we have $x^2-\beta x+2=0$, we find that $\beta$ is not rational.
Let $a_1$ be the trace of the real quadratic integer $\beta$
and $a_2$ its norm.
Then $f$ as in the lemma
is the minimal polynomial of $x$.
As~$x$ is totally non-real,
we find that $\beta^2-8$ is totally negative, which gives upper bounds on the absolute
values of $a_1$ and $a_2$. For all $a_1,a_2$ up to those bounds with $a_1\geq 0$, we checked if $f$ is irreducible
and its roots are Weil $2$-numbers.

For each resulting $a_1$, $a_2$ where $K=\Q[x]$ is a
quartic CM-field that does not contain an imaginary quadratic subfield,
it turned out that
$\Z[x,\overline{x}]$ is its ring of integers.
An algorithm to compute Table~\ref{primitivetable}
is given in~\cite{streng},
but the constants are too large for provenly
correct output.
We know that there are $8$ curves, and know them numerically
to high precision. See Table~\ref{primitivetable}.

We could use the same method (again without proof)
for the remaining cases, where $K=\Q[x]$ does contain
an imaginary quadratic subfield.
Instead, we managed to give a proof using
alternative methods as follows.
As in Section~\ref{sec:degree4},
we get a set of representatives of each isomorphism class of
principally polarized abelian varieties with the endomorphism~$x$,
and we compute approximations with a guaranteed error bound
using interval arithmetic.
Next, let $\phi_1$, $\phi_2$ be the two embeddings
$K\rightarrow\C$
that map $\xi$ to the positive imaginary axis.
Let $K_1\subset K$ be the unique imaginary
quadratic subfield such that $\phi_1$ and $\phi_2$
are equal on $K_1$.
Then $A$ is isogenous to a product
$E_1\times E_2$ of elliptic curve
with CM by
the ring of integers $\mathcal{O}_{K_1}$ of $K_1$
(see~\cite[Theorem 1.3.5]{langcm} or see the
definite case in the discussion of Section~\ref{sec:modpol}).

By considering the corresponding period matrices, we found for
each of them an explicit pair of CM elliptic curves
$E_1\times E_2$ and an integer $n\geq 1$
such that $A$ is $(n,n)$-isogenous to $E_1\times E_2$.
We were lucky enough to have $n\in\{1,2\}$ for each
of our modules.
In the case $n=1$, we simply write down $E_1\times E_2$,
and in the case $n=2$, we used Legendre glueing
(Section~\ref{explicitglue})
to obtain a complete set of candidates, and used
our numerical absolute Igusa invariants to discard all but one.
%%More explicitly, $N_{K/\Q}(\xi)\in\Q$ is a square, say $q^2$ with $q>0$,
%%and we have $K_1=\Q(\xi-q\xi^{-1})\subset K$.
%%The isogeny is given by extending $\C$-linearly an injective
%%$\mathcal{O}_{K_1}$-module homomorphism
%%$(\mathcal{O}_{K_1})^2\rightarrow\mathfrak{a}\cong \Lambda\subset\C^2$
%%to $(\C/\mathcal{O}_{K_1})^2\rightarrow A=\C^2/\Lambda$.
\end{proof}

\subsection{Non-fields}
Finally, we need to consider the case that~$\Q[x]$ is not a domain. The
result is the following lemma.
\begin{lemma}\label{lem:case5}
\label{nonfieldlemma2}
  Let $A$ be a principally polarized abelian surface over $\C$
  and $x:A\rightarrow A$ a $(2,2)$-isogeny.
  If $\Q[x]$ is not a field, then~$A$ is
  isomorphic to one of the following:
  \begin{enumerate}
  \item the $5$ polarized products $E_{-4}\times E_{-4}$, $E_{-4}\times E_{-7}$,
$E_{-4}\times E_{-8}$, $E_{-7}\times E_{-7}$ and $E_{-7}\times E_{-8}$,
  \item the Jacobians of the $4$ curves $C_{-7}$, $C_{-4,-8}$, and $C_{-4,-7}^a$,
      where $a^2=7$.
% $$C_{-4,-8}:y^2  =  (x^2 + 6727)(x^4 +
% 28x^3 + 4774x^2 + 1318492x - 24815903),\quad\mbox{and}$$
%\item the Jacobians of the two conjugate curves
%$$C_{-4,-7}^a:y^2 = x  (x - 2a - 6)  (x^2 - 24a - 48)  (x^2 + (-a - 2)x + 2a + %10),$$
%where $a^2=7$.
\end{enumerate}
  Conversely, each of these abelian surfaces has such an endomorphism~$x$.
\end{lemma}
\begin{proof}
  Given $A$ and $x$, let
  $E_1$, $E_2$, $n$, $\lambda$ be as in Lemma~\ref{nonfieldlemma}.
  As we can interchange $E_1$ and $E_2$, and we can replace $x$ by $\pm x$
  or $\pm x'$,
  we have without loss of generality
  $\left|\Tr\beta_1\right|\geq\left|\Tr\beta_2\right|$,
  $\Tr\beta_1>\Tr\beta_2$,
  and $\mathrm{Im}\beta_1>0$.

  Example \ref{quadraticfields} shows that $\Z[\beta_j]$
  is the maximal order of the imaginary quadratic field of discriminant $-4,-7,$ or $-8$,
  which has class number~$1$.
  Therefore, we get exactly one curve $E_j$ for a given $\beta_j$
  and hence (1) lists exactly all examples with~$n=1$.

  Now assume $n$ is minimal with the properties of Lemma~\ref{nonfieldlemma},
  and suppose that $n\not=1$.
  Let $p|n$ be a prime and let $S=\ker\lambda\cap (E_1\times E_2)[p]$.
  As $\lambda$ is an $(n,n)$-isogeny, its kernel is
  maximally isotropic in $(E_1\times E_2)[n]$
  by Lemma~\ref{isotropicWeil}.
  This implies that
  $S$ is maximally isotropic in $(E_1\times E_2)[p]$.
  By Lemma~\ref{lem:glueingseverywhere},
  it follows that $S$
  is either a product $C_1\times C_2$
  of subgroups $C_j\subset E_j[p]$ of order $p$, or $S$ is the
  graph of an anti-isometry
  $\psi:E_1[p]\rightarrow E_2[p]$.
  Moreover, as
  $(\beta_1,\beta_2)$ induces an endomorphism of $A$,
  it sends $\ker\lambda$
  to itself and hence sends $S$ to $S$.  
  
  Note that $S=C_1\times C_2$ is not possible.
  Indeed, it would imply
  that $E_j/C_j$ has CM by $\Z[\beta_j$] and that
  we have an $(n/p,n/p)$-isogeny
  $$E_1/C_1\times E_2/C_2=(E_1\times E_2)/S\rightarrow(E_1\times E_2)/\ker\lambda=A,$$
  contradicting minimality of~$n$.
  
  Therefore, $S$ is the graph of an anti-isometry
  $\psi:E_1[p]\rightarrow E_2[p]$, and the fact that
  $(\beta_1,\beta_2)$ sends $S$ to itself implies that
  $\beta_2\circ\psi=\psi\circ\beta_1$ holds on~$E_1[p]$.
  
  Let $T= (\Tr\beta_1\bmod p)=(\Tr\beta_2\bmod p)\in\Z/p\Z$.
  We find that we are in one of the following cases:
  \begin{enumerate}
  \item $\beta_1=1+i$, $\beta_2=\pm i\sqrt{2}$, $p=n=2$, $T=0$,
  \item $\beta_1=1+i$, $\beta_2=\frac{1}{2}(-1\pm i\sqrt{7})$, $p=n=3$, $T=-1$,
  \item $\beta_1=1+i$, $\beta_2=-1\pm i$, $p=2$, $n\in\{2,4\}$, $T=0$,
  \item $\beta_1=\frac{1}{2}(1+i\sqrt{7})$,
        $\beta_2=\frac{1}{2}(-1\pm i\sqrt{7})$, $p=n=2$, $T=1$.
  \end{enumerate}

  We can dismiss case (3), and the special
  case of (4) where  $\mathrm{Im}\beta_2<0$, because of the following argument.
  In those cases, we identify $E_1$ with $E_2$ and find that
  $\beta_2-\beta_1$ is a multiple of $2$ in the endomorphism ring of $E_1$,
  hence
  $(\beta_1,\beta_2)$ and $(\beta_1,\beta_1)$ are equal on $(E_1\times E_2)[2]$,
  showing that $(\beta_1, \beta_1)$ induces an endomorphism
  of $(E_1\times E_2)/S$ with both eigenvalues of the
  analytic representation equal to $\beta_1$.
  By Lemma \ref{lem:case3}
  this implies that $(E_1\times E_2)/S\cong E_1\times E_2$
  as a polarized abelian variety,
  contradicting minimality of~$n$.

  In case 1, the difference $\beta_2-\overline{\beta_2}$ is a multiple
  of $2$, so the above argument shows that the
  two possible signs lead to the same set of
  abelian surfaces. Therefore,
  we can restrict to $\Im\beta_2>0$ in this case
  without loss of generality.

  Only the cases (1), (2) and (4) remain and in each of these cases
  except case~(2), we can assume $\Im\beta_2>0$.
  The group
$$E_j[p]=(\C/\Z[\beta_j])[p]=\frac{1}{p}\Z[\beta_j]/\Z[\beta_j],$$
  has a basis $1/p,\beta_j/p$
  and we express $\beta_j$ and $\psi$ by matrices with respect to these bases.
  The identity $\beta_2\psi=\psi\beta_1$ puts some restrictions
  on the matrix coefficients of~$\psi$.
  Moreover, the Weil pairing~$e_p$ on $E_j[p]$
  satisfies $e_p(1/p,\beta_j/p)=\exp(\pm (2\pi i)/p)$, where the sign
  is the sign of the imaginary part of $\beta_j$.
  In particular, the map $\psi$ is an anti-isometry if and only if
  its determinant is $\pm (-1)$ in $\Z/p\Z$.
  We now have a system of equations for the coefficients of $\psi$,
  which we solve.

  For case~(1), we find $2$ solutions, which yield isomorphic
  quotients $(E_1\times E_2)/S$, since the groups $S$
  are mapped to each other by
  $\mathrm{Aut}(E_1\times E_2)=\mathrm{Aut}(E_1)\times \mathrm{Aut}(E_2)$.
  Legendre glueing as in Section~\ref{explicitglue}
  shows that the quotient is~$J(C_{-4,-8})$.
% We have a Sage file with $C(a,b)$
% for a and b as in Legendre glueing for all of these curves.
  For case~(4), we find $1$ solution, which yields $J(C_{-7})$.

  Only case~(2) remains, where $p=3$.
  We find $4$ solutions with $\Im \beta_2>0$
  and $4$ solutions with $\Im\beta_1<0$. Each of these sets
  of $4$ solutions is an $\mathrm{Aut}(E_1\times E_2)$-orbit,
  so that there are at most $2$ quotient abelian surfaces.

In~\cite[Example 6]{kuhn}, Kuhn studies 3-glueings and gives an
example of a family of genus-$2$ curves $C$
  such that $J(C)$ is the image of a $(3,3)$-isogeny
  from a product of elliptic curves.
The curve
$$ 
C: y^2 = (x^3+ax^2+bx+c)(4cx^3+b^2x^2+2bcx+c^2)
$$
is a $3$-glueing of two elliptic curves with
$j$-invariants
$$ 
j(E_1) = \frac{16(972ac^3 - 405b^2c^2 - 216a^2bc^2 + 126ab^3c - 12b^5 - a^2b^4)^3 }{(27c^2 - b^3)^3(27c^2 - 18abc + 4a^3c + 4b^3 - a^2b^2)^2} 
$$
$$
j(E_2) = \frac{256(3b - a^2)^3} {27c^2 - 18abc + 4a^3c + 4b^3 - a^2b^2}.
$$
We use Kuhn's formula and solve for 
  for $j(E_2)= 1728$ and $j(E_1)=-3375$.
%  (we interchange the roles of $E_1$ and $E_2$
%  to make Kuhn's formulas come out nicer).
  This yields equations for $6$ pairwise non-isomorphic curves $C$
  of genus~$2$ of which the Jacobian is the image
  of a $(3,3)$-isogeny from $E_1\times E_2$.

  We claim that this list of $6$ curves contains
  all curves $C$ such that $J(C)$ is $(3,3)$-isogenous
  to $E_1\times E_2$.
  Indeed, the kernel of the isogeny $E_1\times E_2\rightarrow J(C)$
  is the graph of an anti-isometry
  by Lemmas \ref{isotropicWeil} and~\ref{lem:glueingseverywhere}. 
  % and the fact that a polarized abelian variety cannot simultaneously
  % be a polarized product and a Jacobian (by injectivity of the Torelli map).
  There are $24$ anti-isometries from $E_1[3]$ to $E_2[3]$,
  which are partitioned into $6$ orbits of
  composition with $\Z[i]^*$ on the right,
  hence the list of candidate curves
  is complete.
  
  The points in $\A_2$ corresponding to these candidate curves
  are partitioned into two Galois orbits over~$\Q$:
  one of length $2$ and one of length $4$. As we have found only two
  possibilities $A=(E_1\times E_2)/\mathop{\mathrm{graph}}(\psi)$, and
  isomorphism classes of endomorphism rings over $\C$
  are invariant under conjugation over $\Q$,
  our two $A$ cannot lie in the Galois orbit of length $4$,
  so they form the Galois orbit of length $2$.
  They are the $C_{-4,-7}^a$ listed in the lemma.
\end{proof}

\subsection{Summary}
We now state and prove a more explicit version of Theorem~\ref{leq2theorem}.
\begin{theorem}\label{thm:summary}
  The reduced variety $\modulitwo$
  is the 
disjoint
union of the Humbert surface~$H_8$ of discriminant~$8$, 
which is irreducible, and the following
CM points:
\begin{enumerate}
 \item the three points of $\A_2\setminus \M_2$
corresponding to the canonically polarized products
$E_{-4}\times E_{-7}$,
$E_{-4}\times E_{-8}$, and $E_{-7}\times E_{-8}$,
that is, the pairs of $j$-invariants
$\{1728,-3375\}$, $\{1728,8000\}$, $\{-3375,8000\}$,
\item the four points of $\M_2$
corresponding to the curves
$C_{-15}$, $C_{-4,-8}$ and $C_{-4,-7}^a$
$(a^2=7)$
of Table~\ref{tableofcurves}, and
\item the $6$ points of $\M_2$ corresponding to the cases
with $D\not=8$ of
      Table~\ref{primitivetable},
\end{enumerate}
\end{theorem}
\begin{proof}
  The variety $\modulitwo$ is the union of~$H_8$, 
  the Shimura curves of Lemma~\ref{lem:case2},
  and the CM points of Lemmas \ref{lem:case3}, \ref{lem:case4}, and~\ref{lem:case5}.
  We distinguish between points of~$\M_2$, corresponding
  to curves of genus two,
  and points of $\A_2\setminus\M_2$,
  corresponding to canonically polarized products of 
  elliptic curves.

  For any elliptic curve~$E$, the point
  of $\A_2\setminus\M_2$ corresponding to
  $E\times E$ lies on~$H_8$, because
  of the endomorphism
  $$\left(\begin{array}{cc} 1 & 1 \\ 1 & -1\end{array}\right).$$
  The same holds for the point in $\A_2\setminus\M_2$
  corresponding to any product $E\times F$ where $f:E\rightarrow F$
  is a $2$-isogeny, because of 
  $$\left(\begin{array}{cc} 0 & f^\vee \\ f & 0\end{array}\right).$$
  This proves that the first two Shimura curves of Lemma~\ref{lem:case2} lie
  on~$H_8$, as well as all points in $\A_2\setminus\M_2$
  under consideration, except those listed in~(1).
  The points listed in (1) do not lie on~$H_8$,
  as the corresponding endomorphism rings are direct products of
  rings without~$\sqrt{2}$.

  We evaluate a defining equation for~$H_8$,
  provided by David Gruenewald,
% This is a defining equation in terms of absolute Igusa invariants,
% provided to us by email, and easily checked to be correct
% using the defining equations on his web page.
  in those points in $\M_2$ for which we know proven
  values of the Igusa invariants. This allows us to decide
  whether the point is on~$H_8$, and (2) lists exactly the
  cases where the point does not lie on~$H_8$.
  We use the same trick for the Shimura curves of which an
  open piece lies in~$\M_2$,
  which proves that the remaining three Shimura curves
  lie on~$H_8$.

  This leaves only the CM points of Table~\ref{primitivetable},
  where the endomorphism ring
  equals
  the maximal order of the
  quartic field~$K$~\cite[Theorem~1.3]{vanwamelen}.
  The unique real quadratic subfield of~$K$
  has discriminant~$D$,
  with~$D$ listed in the table. In particular, the point is on~$H_8$
  if and only if~$D=8$, that is, if and only if the point
  is not listed in~(3).
\end{proof}

\bibliographystyle{plain}
\bibliography{bls}

\def\cprime{$'$}
\begin{thebibliography}{10}

\bibitem{bass4}
Hyman Bass.
\newblock Torsion free and projective modules.
\newblock {\em Trans. Amer. Math. Soc.}, 102:319--327, 1962.

\bibitem{birkenhakelange}
Christina Birkenhake and Herbert Lange.
\newblock {\em Complex abelian varieties}, volume 302 of {\em Grundlehren der
  mathematischen Wissenschaften}.
\newblock Springer, second edition, 2004.

\bibitem{magma}
Wieb Bosma, John Cannon, and Catherine Playoust.
\newblock The {M}agma algebra system {I}: {T}he user language.
\newblock {\em J. Symbolic Comput.}, 24(3-4):235--265, 1997.
\newblock Computational algebra and number theory (London, 1993).

\bibitem{bouyer-streng}
Florian Bouyer and Marco Streng.
\newblock Examples of {CM} curves of genus two defined over the reflex field.
\newblock To appear, 2012.

\bibitem{brokerlauter}
Reinier Br{\"o}ker and Kristin Lauter.
\newblock Modular polynomials for genus $2$.
\newblock {\em London Mathematical Society, Journal of Computation and
  Mathematics}, 12:326--339, 2009.

\bibitem{hehcc5}
Gerhard Frey and Tanja Lange.
\newblock Varieties over special fields.
\newblock In Henri Cohen, Gerhard Frey, Roberto Avanzi, Christophe Doche, Tanja
  Lange, Kim Nguyen, and Frederik Vercauteren, editors, {\em Handbook of
  elliptic and hyperelliptic curve cryptography}, pages 87--113. Chapman \&
  Hall/CRC, 2006.

\bibitem{goren-lauter}
Eyal~Z. Goren and Kristin Lauter.
\newblock Genus 2 curves with complex multiplication.
\newblock {\em Int Math Res Notices}, 2012(5):1068 -- 1142, 2012.

\bibitem{gruenewalddata}
David Gruenewald.
\newblock Humbert surface data, 2008.
\newblock \url{http://echidna.maths.usyd.edu.au/~davidg/thesis.html}.

\bibitem{gruenewald}
David Gruenewald.
\newblock Computing {H}umbert surfaces and applications.
\newblock In {\em Arithmetic, geometry, cryptography and coding theory 2009},
  volume 521 of {\em Contemp. Math.}, pages 59--69. Amer. Math. Soc.,
  Providence, RI, 2010.

\bibitem{hayashidanishi}
Tsuyoshi Hayashida and Mieo Nishi.
\newblock Existence of curves of genus two on a product of two elliptic curves.
\newblock {\em J. Math. Soc. Japan}, 17(1):1--16, 1965.

\bibitem{igusa}
Jun-Ichi Igusa.
\newblock Arithmetic variety of moduli for genus two.
\newblock {\em Annals of Mathematics}, 72(3):612--649, 1960.

\bibitem{Kani}
Ernst Kani.
\newblock The number of curves of genus two with elliptic differentials.
\newblock {\em J. reine angew. Math.}, 485:93--121, 1997.

\bibitem{echidna-data}
David~R. Kohel.
\newblock {\em ECHIDNA: Databases for elliptic curves and higher dimensional
  analogues}, 2008.
\newblock \url{http://echidna.maths.usyd.edu.au/}.

\bibitem{kuhn}
Robert~M. Kuhn.
\newblock Curves of genus $2$ with split {J}acobian.
\newblock {\em Transactions of the American Mathematical Society},
  307(1):41--49, 1988.

\bibitem{langcm}
Serge Lang.
\newblock {\em Complex Multiplication}, volume 255 of {\em Grundlehren der
  mathematischen Wissenschaften}.
\newblock Springer, 1983.

\bibitem{LaANT}
Serge Lang.
\newblock {\em Algebraic Number Theory}, volume 110 of {\em Graduate Texts in
  Mathematics}.
\newblock Springer, second edition, 1994.

\bibitem{lauter-viray}
Kristin Lauter and Bianca Viray.
\newblock An arithmetic intersection formula for denominators of {I}gusa class
  polynomials.
\newblock preprint, \href{http://arxiv.org/pdf/1210.7841.pdf}{arXiv:1210.7841},
  2012.

\bibitem{hehcc17}
Reynald Lercier, David Lubicz, and Frederik Vercauteren.
\newblock Point counting on elliptic and hyperelliptic curves.
\newblock In Henri Cohen, Gerhard Frey, Roberto Avanzi, Christophe Doche, Tanja
  Lange, Kim Nguyen, and Frederik Vercauteren, editors, {\em Handbook of
  elliptic and hyperelliptic curve cryptography}, pages 407--454. Chapman \&
  Hall/CRC, 2006.

\bibitem{mestre}
Jean-Fran{\c{c}}ois Mestre.
\newblock Construction de courbes de genre {$2$} {\`{a}} partir de leurs
  modules.
\newblock In {\em Effective methods in algebraic geometry ({C}astiglioncello,
  1990)}, volume~94 of {\em Progr. Math.}, pages 313--334. Birkh\"auser Boston,
  Boston, MA, 1991.

\bibitem{milne}
James Milne.
\newblock Abelian varieties.
\newblock In G.~Cornell and J.H. Silverman, editors, {\em Arithmetic Geometry},
  pages 103--150. Springer-Verlag, New York, 1986.

\bibitem{Ser85}
Jean-Pierre Serre.
\newblock Rational points on curves over finite fields.
\newblock {\em Notes by F. Gouvea of lectures at Harvard University}, 1985.

\bibitem{shimura-unitary}
Goro Shimura.
\newblock Arithmetic of unitary groups.
\newblock {\em Ann. of Math. (2)}, 79:369--409, 1964.

\bibitem{silverman}
Joseph~H. Silverman.
\newblock {\em The Arithmetic of Elliptic Curves}, volume 106 of {\em Graduate
  Texts in Mathematics}.
\newblock Springer, 1986.

\bibitem{spallek}
Anne-Monika Spallek.
\newblock {\em Kurven vom {G}eschlecht {$2$} und ihre {A}nwendung in
  {P}ublic-{K}ey-{K}ryptosystemen}.
\newblock PhD thesis, Institut f{\"u}r Experimentelle Mathematik,
  Universit{\"a}t GH Essen, 1994.

\bibitem{streng}
Marco Streng.
\newblock Computing {I}gusa class polynomials.
\newblock Accepted for publication by Mathematics of Computation,
  arXiv:0903.4766, 2012.

\bibitem{vdgeer}
Gerard van~der Geer.
\newblock {\em Hilbert modular surfaces}, volume~16 of {\em Ergebnisse der
  Mathematik und ihrer Grenzgebiete (3)}.
\newblock Springer-Verlag, Berlin, 1988.

\bibitem{vanwamelen}
Paul van Wamelen.
\newblock Examples of genus two {CM} curves defined over the rationals.
\newblock {\em Math. Comp.}, 68(225):307--320, 1999.

\bibitem{wamelen-proving}
Paul van Wamelen.
\newblock Proving that a genus {$2$} curve has complex multiplication.
\newblock {\em Math. Comp.}, 68(228):1663--1677, 1999.

\end{thebibliography}
\end{document}